\documentclass[aop]{imsart}
\usepackage{amsmath}
\usepackage{amsfonts}
\usepackage{graphicx}
\usepackage{subcaption} 
\usepackage[justification=centering]{caption}
\usepackage{amssymb}
\usepackage{tikz}
\usepackage{amsthm }
\usepackage{graphicx}
\usepackage{enumerate}
\usepackage{cancel}
\usepackage{epstopdf}
\usepackage[margin=1in]{geometry}
\RequirePackage{amsthm,amsmath}
\RequirePackage[numbers]{natbib}
\RequirePackage[colorlinks,citecolor=blue,urlcolor=blue]{hyperref}

\startlocaldefs
\newtheorem{Thm}{Theorem}[section]
\newtheorem{Lem}{Lemma}[section]
\newtheorem{Pro}{Proposition}[section]
\newtheorem{Cor}{Corollary}[section]

\theoremstyle{definition}
\newtheorem{Rem}{Remark}[section]

\newcommand{\ConvD}{\overset{d}{\rightarrow}}
\newcommand{\ConvFDD}{\overset{f.d.d.}{\longrightarrow}}
\newcommand{\B}{\mathrm{B}}
\newcommand{\Var}{\mathrm{Var}}
\newcommand{\E}{\mathbb{E}}
\newcommand{\Del}{\boldsymbol{\Delta}}
\endlocaldefs

\begin{document}

\begin{frontmatter}
\title{Behavior of the generalized Rosenblatt process \\at extreme critical exponent values}
\runtitle{Behavior of the generalized Rosenblatt at critical exponent values}

%\begin{keyword}[class=MSC]
%\kwd[Primary]{60G18}
%\kwd[secondary]{60F05}
%\end{keyword}

%\begin{keyword}
%\kwd{}
%\kwd{}
%\end{keyword}

\begin{aug}
\author{\fnms{Shuyang} \snm{Bai}\thanksref{m1}\ead[label=e1]{bsy9142@gmail.com}} and
\author{\fnms{Murad} S. \snm{Taqqu}\corref{tq}\thanksref{t1,m1}\ead[label=e2]{murad@bu.edu}}

\thankstext{t1}{Corresponding author}
\runauthor{Bai, S.  and  Taqqu, M.S.}

\affiliation{Boston University\thanksmark{m1}}

\address{Department of Mathematics and Statistics\\
Boston University\\
111 Cumminton Mall\\
Boston, MA, 02215, US\\
\printead{e1}
}
\address{Department of Mathematics and Statistics\\
Boston University\\
111 Cumminton Mall\\
Boston, MA, 02215, US\\
\printead{e2}
}
\end{aug}

\begin{abstract}
The generalized Rosenblatt process  is obtained by replacing the single critical exponent characterizing the Rosenblatt process by two different  exponents living in the interior of a triangular  region.  
What happens to that generalized Rosenblatt process as these critical exponents approach the boundaries of the triangle? We show by two different methods that on each of the two  symmetric boundaries, the limit is non-Gaussian. On the third boundary, the limit is Brownian motion. The rates of convergence to these boundaries are also given.  The situation is particularly delicate as one approaches the corners of the triangle, because the limit process will depend on how these corners are approached. All limits are in the sense of weak convergence in $C[0,1]$. These limits cannot be strengthened to convergence in $L^2(\Omega)$.

\end{abstract}

\begin{keyword}[class=MSC]
\kwd{60F05}
\kwd{60K35}
\end{keyword}

\begin{keyword}
\kwd{Long memory}
\kwd{Self-similar processes}  
\kwd{Rosenblatt processes}
\kwd{Generalized Rosenblatt processes} 
\end{keyword}

\end{frontmatter}

\section{Introduction}

\citet{maejima:tudor:2012:selfsimilar}  considered recently the following process defined through a second-order Wiener-It\^o integral: 
\begin{equation}\label{eq:ori gen Ros proc}
Z_{\gamma_1,\gamma_2}(t)=  A~ \int_{\mathbb{R}^2}' \left[\int_0^t (s-x_1)_+^{\gamma_1}(s-x_2)_+^{\gamma_2} ds \right] B(dx_1) B(dx_2),
\end{equation}
where $A\neq 0$ is a constant, $B(\cdot)$ is a Brownian random measure, the prime $'$ indicates the \emph{exclusion} of the diagonals $x_1=x_2$ in the double stochastic integral, and the exponents $\gamma_1,\gamma_2$ live in the following open triangular region (see Figure \ref{fig:main}): 
\begin{equation}\label{eq:region}
\Del=\{(\gamma_1,\gamma_2): ~-1<\gamma_1<-1/2,~ -1<\gamma_2<-1/2,  ~\gamma_1+\gamma_2>-3/2\}. 
\end{equation}
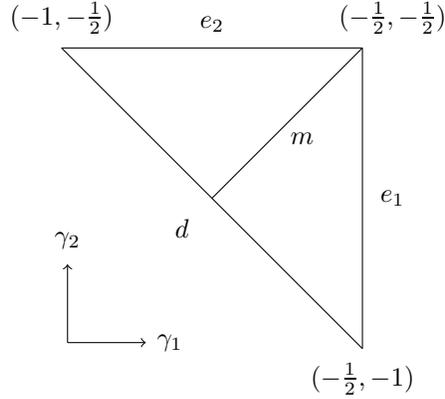
\begin{figure}
\centering
\begin{tikzpicture}[scale=8]
\draw[->] (-0.99,-0.99) -- (-0.86,-0.99);
\draw[->] (-0.99,-0.99) -- (-0.99,-0.86);
\node at (-0.82,-0.99) {\small $\gamma_1$};
\node at (-0.99,-0.82) {\small $\gamma_2$};
\draw (-1/2,-1/2)-- (-1,-1/2);
\draw (-1,-1/2)-- (-1/2,-1);
\draw (-1/2,-1)-- (-1/2,-1/2);
\draw (-3/4,-3/4)-- (-1/2,-1/2);
\node at (-0.45, -0.45) {\small $(-\frac{1}{2},-\frac{1}{2})$};
\node at (-1/2,-1.05) {\small $(-\frac{1}{2},-1)$};
\node at (-1, -0.45) {\small $(-1,-\frac{1}{2})$};
\node at (-0.6,-0.65) {\small $m$};
\node at (-0.45,-0.75) {\small $e_1$};
\node at (-0.75,-0.46) {\small $e_2$};
\node at (-0.8,-0.8) {\small $d$};
\end{tikzpicture}
\caption{Region $\Del$ defined in (\ref{eq:region}).
\\{\small
 The three edges of the triangle are named $e_1,e_2$ and $d$ (diagonal), while the middle line segment (symmetric axis) is named $m$.}}\label{fig:main}
\end{figure}
This ensures that the integrand in (\ref{eq:ori gen Ros proc}) is in $L^2(\mathbb{R}^2)$, and hence the process $Z_{\gamma_1,\gamma_2}(t)$ is well-defined (see Theorem 3.5 and Remark 3.1 of \citet{bai:taqqu:2013:3-generalized}).

We shall call $Z_{\gamma_1,\gamma_2}(t)$ a \emph{generalized Rosenblatt process}. The \emph{Rosenblatt process} $Z_{\gamma}(t)$ (\citet{taqqu:1975:weak}) becomes the special case
\begin{equation}\label{eq:Rosenblatt proc}
Z_{\gamma}(t)=Z_{\gamma,\gamma}(t), \qquad  -3/4<\gamma<-1/2.
\end{equation}
Recent studies on the Rosenblatt process $Z_\gamma(t)$ include
  \citet{tudor:Viens:2009:variations}, \citet{bardet:tudor:2010:wavelet},  \citet{arras:2013:white}, \citet{maejima:tudor:2013:distribution},  \citet{veillette:taqqu:2013:properties}  and \citet{bojdecki:Gorostiza:2013:intersection}. The Rosenblatt and the generalized Rosenblatt processes are of interest because they are the simplest extension to the non-Gaussian world of the Gaussian fractional Brownian motion.

Fractional Brownian motion $B_H(t)$, $1/2<H<1$ is defined through a single Wiener-It\^o (or Wiener) integral:
\[
B_H(t)= C\int_{\mathbb{R}}\left[ \int_0^t  (s-x)_+^{H-3/2} ds\right] B(dx),
\]
and has covariance
\begin{equation}\label{eq:fbm cov}
\E B_{H}(s)B_{H}(t)=\frac{C'}{2}\left( |s|^{2H}+ |t|^{2H}-|s-t|^{2H}\right),
\end{equation}
where $C$ and $C'$ are two related constants. Fractional Brownian motion reduces to Brownian motion if one sets $H=1/2$ in (\ref{eq:fbm cov}). Fractional Brownian motion has stationary increments and, for any $1/2<H<1$,  these increments have a covariance which decreases slowly as the lag increases. This slow decay is often referred to as \emph{long memory} or \emph{long-range dependence}. Fractional Brownian motion is also self-similar with \emph{self-similarity parameter (Hurst index)} $H$, that is, $B_H(\lambda t)$ has the same finite-dimensional distributions as $\lambda^H B_H(t)$ for any $\lambda>0$. It follows from \citet{bai:taqqu:2013:3-generalized} that the generalized Rosenblatt process $Z_{\gamma_1,\gamma_2}(t)$ is also self-similar with stationary increments with  self-similarity parameter  
\begin{equation}\label{eq:hurst}
H=\gamma_1+\gamma_2+2\in (1/2,1).
\end{equation}
We get $1/2<H<1$ because $\gamma_1,\gamma_2<-1/2$ imply $H<1$ and $\gamma_1+\gamma_2>-3/2$ implies $H>1/2$. 

Fractional Brownian motion and the generalized Rosenblatt process $Z_{\gamma_1,\gamma_2}(t)$ belong  to a broad class of self-similar processes with stationary increments defined on a Wiener chaos called \emph{generalized Hermite processes}. 
The generalized Hermite processes appear as limits in various types of non-central limit theorems involving Volterra-type nonlinear process.  In particular, the generalized Rosenblatt process $Z_{\gamma_1,\gamma_2}(t)$ can arise as limit when considering a quadratic form involving two long-memory linear processes with different memory parameters. See \citet{bai:taqqu:2013:3-generalized,bai:taqqu:2014:4-convergence,bai:taqqu:2014:5-impact} for details.

It will be convenient to express the generalized Rosenblatt process as follows, 
\begin{equation}\label{eq:gen Rosenblatt proc}
Z_{\gamma_1,\gamma_2}(t)=  \frac{A}{2}~ \int_{\mathbb{R}^2}' \left[\int_0^t [(s-x_1)_+^{\gamma_1}(s-x_2)_+^{\gamma_2}+(s-x_1)_+^{\gamma_2}(s-x_2)_+^{\gamma_1}] ds \right]B(dx_1) B(dx_2),
\end{equation}
where we replaced the kernel $A\int_0^t (s-x_1)_+^{\gamma_1}(s-x_2)_+^{\gamma_2}ds$
by its symmetrized version. The process $Z_{\gamma_1,\gamma_2}(t)$ remains invariant under such a modification.

The goal of this paper is to study the distributional behavior of the standardized $Z_{\gamma_1,\gamma_2}(t)$ (where $A$ in 
(\ref{eq:gen Rosenblatt proc}) is chosen so that $\Var[Z_{\gamma_1,\gamma_2}(1)]=1$), as $(\gamma_1,\gamma_2)$ approaches the boundaries of the region $\Del$ defined in (\ref{eq:region}).

We show that on the diagonal boundary $d$, the limit is Brownian motion.  On each of the two  symmetric boundaries $e_1$ and $e_2$ of $\Del$, the limit is non-Gaussian: it is a fractional Brownian motion times an independent Gaussian random variable.  We give two different proofs of this convergence, one based on the method of moments, and one which provides more intuitive insight.  We also give the rate of convergence to the marginal distribution in the preceding two cases.

The situation at the corners is particularly delicate. At the corner $(\gamma_1,\gamma_2)=(-1/2,-1/2)$, the limit process is a  linear combination of two  independent degenerate chi-square processes. At the other two corners, the limit is a linear combination of two processes: a Brownian motion and the product of another Brownian motion times an independent Gaussian random variable. These linear combinations, which depend on the direction at which the critical exponents approach the corners, will be given explicitly.

 We also show that  the convergences  mentioned cannot be strengthened from weak convergence to $L^2(\Omega)$ convergence, nor even to convergence in probability.

%when $\gamma_1\rightarrow -1/2$ and $\gamma_2\rightarrow \gamma$ with $-1<\gamma<-1/2$, the limit is a product of a fractional Brownian motion with an independent normal random variable;  when $\gamma_1+\gamma_2\rightarrow -3/2$ and both $\gamma_1,\gamma_2$ are away from $-1/2$, the limit is  Brownian motion. At the corners the situation is delicate and the limit depends on the direction at which the corners are approached. At the corner $(\gamma_1,\gamma_2)=(-1/2,-1)$ or $(-1,-1/2)$, the limit is a linear combination of a Brownian motion $B_1(t)$ with $W\cdot B_2(t)$, where $B_1(t)$ and $B_2(t)$ are Brownian motions and $W$ is a Gaussian random variable, $B_1(t)$, $B_2(t)$ and $W$ being mutually independent;at the corner $(\gamma_1,\gamma_2)=(-1/2,-1/2)$ or $(-1,-1/2)$, the limit is a  linear combination of $t\cdot X_1$ and $t \cdot X_2$, where $X_1$ and $X_2$ are two standardized $\chi_1^2$ random variables;  The weights of these linear combinations  at the corners depend on the direction $(\gamma_1,\gamma_2)$ approaches the corners.

The paper is organized as follows. In Section \ref{sec:main}, we state the main results with proofs  in Section \ref{sec:proof}. In the following three sections, we provide some additional results: showing that  $L^2(\Omega)$ convergence cannot hold,  establishing the rate of marginal convergence on the boundaries $d$, $e_1$ and $e_2$, and giving an alternate proof of the convergence on the boundaries $e_1$ and $e_2$.

\section{Main results}\label{sec:main}

 In the following theorems, we let $\Rightarrow$ denote weak convergence in the space $C[0,1]$ with uniform metric. The multiplicative factor $A$  in (\ref{eq:gen Rosenblatt proc}) is chosen so that $\Var[Z_{\gamma_1,\gamma_2}(1)]=1$. See (\ref{eq:A}) below for an explicit expression.

We focus first on results concerning  the behavior of $Z_{\gamma_1,\gamma_2}(t)$ as $(\gamma_1,\gamma_2)$ approaches the boundary of $\Del$ in (\ref{eq:region}), excluding the corners. 
Theorem \ref{Thm:main diag} involves convergence to the diagonal edge $d$ of $\Del$, where the limit is Brownian motion.  See Figure \ref{fig:thm diag}.
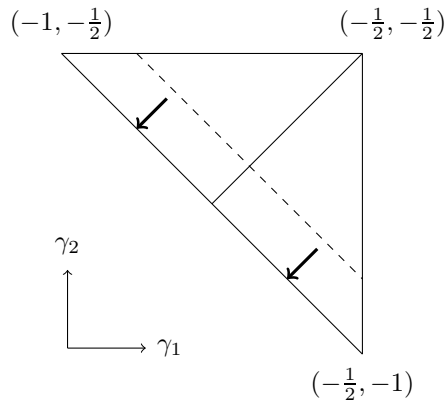
\begin{figure}[t,h]
\centering
\begin{tikzpicture}[scale=8]
\draw[->] (-0.99,-0.99) -- (-0.86,-0.99);
\draw[->] (-0.99,-0.99) -- (-0.99,-0.86);
\node at (-0.82,-0.99) {\small $\gamma_1$};
\node at (-0.99,-0.82) {\small $\gamma_2$};
\draw (-1/2,-1/2)-- (-1,-1/2);
\draw (-1,-1/2)-- (-1/2,-1);
\draw (-1/2,-1)-- (-1/2,-1/2);
\draw (-3/4,-3/4)-- (-1/2,-1/2);
\node at (-0.45, -0.45) {\small $(-\frac{1}{2},-\frac{1}{2})$};
\node at (-1/2,-1.05) {\small $(-\frac{1}{2},-1)$};
\node at (-1, -0.45) {\small $(-1,-\frac{1}{2})$};
\draw[->,very thick] (-1.5+0.875+0.05,-0.875+0.05)--(-1.5+0.875,-0.875);
\draw[->,very thick] (-0.875+0.05,-1.5+0.875+0.05)--(-0.875,-1.5+0.875);
\draw[dashed] (-0.875,-1/2)--(-1/2,-0.875);
\end{tikzpicture}
\caption{Illustration of limit taking in Theorem \ref{Thm:main diag}}\label{fig:thm diag}
\end{figure}
\begin{Thm}\label{Thm:main diag}
Let $Z_{\gamma_1,\gamma_2}(t)$, $(\gamma_1,\gamma_2)\in \Del$, be defined in (\ref{eq:gen Rosenblatt proc}) with $A=A(\gamma_1,\gamma_2)$ in (\ref{eq:A}). 
When $\gamma_1+\gamma_2\rightarrow -3/2$ with $\gamma_1,\gamma_2>-1+\epsilon$ for arbitrarily fixed $\epsilon>0$, we have
\begin{equation}\label{eq:boundary two}
Z_{\gamma_1,\gamma_2}(t) \Rightarrow   B(t),
\end{equation}
where $B(t)$ is a standard Brownian motion.
\end{Thm}
One has $\gamma_1+\gamma_2=-3/2$ all through the diagonal $d$. The corners of the triangle are excluded by the requirement $\gamma_1,\gamma_2>-1+\epsilon$. Convergence to Brownian motion in (\ref{eq:boundary two}) is  expected heuristically since the self-similarity parameter $H=\gamma_1+\gamma_2+2\rightarrow 1/2$ (see (\ref{eq:hurst})), and $1/2$ is the self-similarity parameter of Brownian motion.

The next Theorem \ref{Thm:main edge}  involves convergence to either one of the two sides $e_1$ and $e_2$ of $\Del$. The vertical side $e_1$ and the horizontal side $e_2$ are parameterized respectively by $(-1/2,\gamma)$ and $(\gamma,-1/2)$ where $-1<\gamma<-1/2$. See Figure \ref{fig:thm edge}.
\begin{Thm}\label{Thm:main edge}
Let $Z_{\gamma_1,\gamma_2}(t)$, $(\gamma_1,\gamma_2)\in \Del$, be defined in (\ref{eq:gen Rosenblatt proc}) with $A=A(\gamma_1,\gamma_2)$ in (\ref{eq:A}). 
When $(\gamma_1,\gamma_2)\rightarrow (-1/2,\gamma)$ or $(\gamma_1,\gamma_2)\rightarrow(\gamma,-1/2)$, where $-1<\gamma<-1/2$ , we have
\begin{equation}\label{eq:boundary one}
Z_{\gamma_1,\gamma_2}(t) \Rightarrow  W  B_{\gamma+3/2}(t),
\end{equation}
where $B_{\gamma+3/2}(t)$ is a standard fractional Brownian motion with self-similarity parameter $\gamma+3/2$, and $W$ is a standard normal random variable which is independent of $B_{\gamma+3/2}(t)$.
\end{Thm}
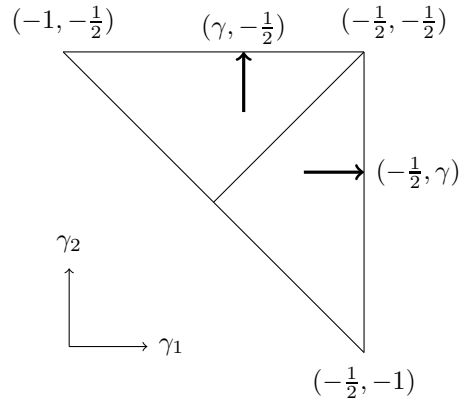
\begin{figure}[t,h]
\centering
\begin{tikzpicture}[scale=8]
\draw[->] (-0.99,-0.99) -- (-0.86,-0.99);
\draw[->] (-0.99,-0.99) -- (-0.99,-0.86);
\node at (-0.82,-0.99) {\small $\gamma_1$};
\node at (-0.99,-0.82) {\small $\gamma_2$};
\draw (-1/2,-1/2)-- (-1,-1/2);
\draw (-1,-1/2)-- (-1/2,-1);
\draw (-1/2,-1)-- (-1/2,-1/2);
\draw (-3/4,-3/4)-- (-1/2,-1/2);
\node at (-0.45, -0.45) {\small $(-\frac{1}{2},-\frac{1}{2})$};
\node at (-1/2,-1.05) {\small $(-\frac{1}{2},-1)$};
\node at (-1, -0.45) {\small $(-1,-\frac{1}{2})$};
\draw[->,very thick] (-0.6,-0.7)--(-1/2,-0.7);
\draw[->,very thick] (-0.7,-0.6)--(-0.7,-0.5);
\node at (-0.7,-0.46) {\small $(\gamma,-\frac{1}{2})$};
\node at (-0.41,-0.7) {\small $(-\frac{1}{2},\gamma)$};
\end{tikzpicture}
\caption{Illustration of limit taking in Theorem \ref{Thm:main edge}}\label{fig:thm edge}
\medskip
\end{figure}

\begin{Rem}\label{Rem:physical}
The convergence (\ref{eq:boundary one}) is more involved since $W B_{\gamma+3/2}(t)$ is a self-similar process with stationary increments having self-similarity parameter $H=\gamma+3/2\in (1/2,1)$, and hence displays long-range dependence.  This convergence may be understood heuristically as follows: $Z_{\gamma_1,\gamma_2}(t)$ in (\ref{eq:ori gen Ros proc}) can be regarded as an integrated process of a long-range dependent bilinear moving average of white noise. This bilinear moving average involves a double summation. As the exponent $\gamma_1\rightarrow -1/2$, the corresponding summation yields a term which is extremely persistent, so that it behaves like a frozen Gaussian variable which is independent of the fractional noise defined through the other summation. 
\end{Rem}
\begin{Rem}
Although intuitively the generalized Rosenblatt processes $Z_{\gamma_1,\gamma_2}(t)$ in (\ref{eq:ori gen Ros proc}) form a richer class than the Rosenblatt process $Z_{\gamma}(t)$ in (\ref{eq:Rosenblatt proc}), they are both self-similar with stationary increments, and hence have the same covariance (\ref{eq:fbm cov}) when $2\gamma=\gamma_1+\gamma_2$. To show that they are different processes, one can compare the higher moments, as was done in  \citet{bai:taqqu:2013:6-structure}. 
The convergence (\ref{eq:boundary one}) provides another evidence that there are values of $(\gamma_1,\gamma_2)$ for which $Z_{\gamma_1,\gamma_2}(t)$ is different from $Z_{\gamma}(t)$. Indeed the limit  $W  B_{\gamma+3/2}(t)$  has a symmetric marginal distribution (the so-called  product-normal distribution), while the marginal distribution of the Rosenblatt process $Z_{\gamma}(t)$  is  skewed with a nonzero third cumulant (see (10) and (12) of \citet{veillette:taqqu:2013:properties}, or set $\gamma_1=\gamma_2=\gamma$ in (\ref{eq:mu_3}) below). 
\end{Rem}

Note that in Theorem \ref{Thm:main diag} and \ref{Thm:main edge}, we exclude the three corners $(\gamma_1,\gamma_2)=(-1/2,-1/2)$, $(-1,-1/2)$ and $(-1/2,-1)$. It turns out that the limit behavior of $Z_{\gamma_1,\gamma_2}(t)$ at these corners depends on the direction these corners are approached. 
Due to the symmetry of $Z_{\gamma_1,\gamma_2}(t)$ in $(\gamma_1,\gamma_2)$, it is sufficient to focus on the case $\gamma_1\ge \gamma_2$, that is, we focus on the subregion of $\Del$ in (\ref{eq:region}) delimited by line segments $e_1,d$ and $m$ in Figure \ref{fig:thm corner (-1/2,-1)}.

Consider first the corner $(\gamma_1,\gamma_2)=(-1/2,-1)$. We will approach it through the line 
\[
\gamma_2=\frac{1}{\rho-1}(\gamma_1+1/2)-1, 
\]
which can also be expressed as
\[
\frac{\gamma_1+\gamma_2+3/2}{\gamma_2+1}=\rho.
\]
The line passes through the corner $(-1/2,-1)$ and has a negative slope of $1/(\rho-1)$, $0\le \rho\le  1$. 
 See Figure \ref{fig:thm corner (-1/2,-1)}. When $\rho=0$, the line coincides with the diagonal edge $d$ of the triangle $\Del$, which has slope $-1$. When $\rho=1$, the line coincides with the vertical side $e_1$ of $\Del$, which has slope $-\infty$.
\begin{figure}[t,h]
\centering
\begin{tikzpicture}[scale=8]
\draw (-1/2,-1/2)-- (-1,-1/2);
\draw (-1,-1/2)-- (-1/2,-1);
\draw (-1/2,-1)-- (-1/2,-1/2);
\draw (-3/4,-3/4)-- (-1/2,-1/2);
\node at (-0.45, -0.45) {\small $(-\frac{1}{2},-\frac{1}{2})$};
\node at (-1/2,-1.05) {\small $(-\frac{1}{2},-1)$};
\node at (-1, -0.45) {\small $(-1,-\frac{1}{2})$};
\draw[->,very thick] (-1/2-0.05,-1+0.1)--(-1/2-0.013,-1+0.035);
\node at (-0.6,-0.65) {\small $m$};
\node at (-0.45,-0.75) {\small $e_1$};
\node at (-0.75,-0.46) {\small $e_2$};
\node at (-0.68,-0.89) {\small $d$};
\node at (-0.845,-1.01) {\small $\gamma_1$};
\node at (-1.01,-0.85) {\small $\gamma_2$};
\draw[->] (-1.01,-1.01) -- (-1.01,-0.87);
\draw[->] (-1.01,-1.01) -- (-0.87,-1.01);
\node at (-0.87,-0.77) {\small $(-\frac{3}{4},-\frac{3}{4})$};
\end{tikzpicture}
\caption{Illustration of limit taking in Theorem \ref{Thm:main corner (-1/2,-1)}}\label{fig:thm corner (-1/2,-1)}
\end{figure}
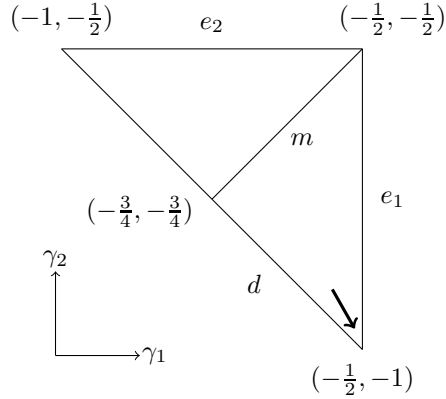

\begin{Thm}[The corner $(\gamma_1,\gamma_2)=(-1/2,-1)$]\label{Thm:main corner (-1/2,-1)} \ 
\\
Let $Z_{\gamma_1,\gamma_2}(t)$, $(\gamma_1,\gamma_2)\in \Del$, be defined in (\ref{eq:gen Rosenblatt proc}) with $A=A(\gamma_1,\gamma_2)$ in (\ref{eq:A}).  Suppose that $\gamma_1\ge \gamma_2$.
If $(\gamma_1,\gamma_2)\rightarrow (-1/2,-1)$ in  such a way that
\begin{equation}\label{eq:rho limit corner 1}
\frac{\gamma_1+\gamma_2+3/2}{\gamma_2+1}=1+\frac{\gamma_1+1/2}{\gamma_2+1}\rightarrow \rho \in [0,1],
\end{equation}
then
\begin{equation}\label{eq:X_rho}
Z_{\gamma_1,\gamma_2}(t)\Rightarrow X_\rho(t):= \rho^{1/2}  W B(t)+(1-\rho)^{1/2}B'(t),
\end{equation}
where $W$ is a standard normal random variable, $B(t)$ and $B'(t)$ are standard Brownian motions, and $W$, $B(t)$ and $B'(t)$ are independent.
\end{Thm}
\begin{Rem}
In Theorem \ref{Thm:main corner (-1/2,-1)}, the limit $X_\rho(t)$ is an independent linear combination of the two limits obtained  in Theorem \ref{Thm:main edge} and \ref{Thm:main diag} (edges $e_1$ and $d$), after setting   $\gamma=-1$ in Theorem \ref{Thm:main edge}. Note that since $\gamma+3/2=-1+3/2=1/2$, the fractional Brownian motion $B_{\gamma+3/2}(t)$ in Theorem \ref{Thm:main edge} becomes Brownian motion $B(t)$.
\end{Rem}

Consider now the corner $(\gamma_1,\gamma_2)=(-1/2,-1/2)$. We will approach it
through the line 
\[
\gamma_2=\frac{1}{\rho}(\gamma_1+1/2)-1/2, 
\]
which passes through it and has a positive slope of $1/\rho$, $0\le \rho\le 1$. See Figure \ref{fig:thm corner (-1/2,-1/2)}. When $\rho=0$, the line coincides with the vertical side $e_1$ of $\Del$, which has slope $+\infty$.  When $\rho=1$, the line coincides with the middle line $m$, which has slope $1$.

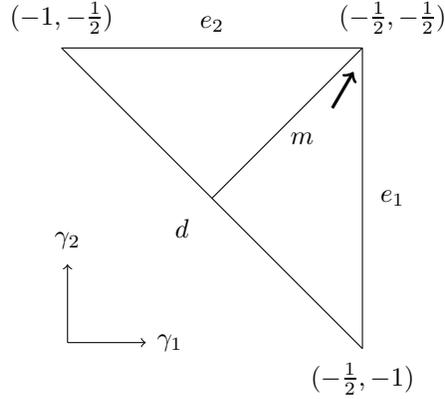
\begin{figure}
\centering
\begin{tikzpicture}[scale=8]
\draw[->] (-0.99,-0.99) -- (-0.86,-0.99);
\draw[->] (-0.99,-0.99) -- (-0.99,-0.86);
\node at (-0.82,-0.99) {\small $\gamma_1$};
\node at (-0.99,-0.82) {\small $\gamma_2$};
\draw (-1/2,-1/2)-- (-1,-1/2);
\draw (-1,-1/2)-- (-1/2,-1);
\draw (-1/2,-1)-- (-1/2,-1/2);
\draw (-3/4,-3/4)-- (-1/2,-1/2);
\node at (-0.45, -0.45) {\small $(-\frac{1}{2},-\frac{1}{2})$};
\node at (-1/2,-1.05) {\small $(-\frac{1}{2},-1)$};
\node at (-1, -0.45) {\small $(-1,-\frac{1}{2})$};
\draw[->,very thick] (-1/2-0.05,-1/2-0.1)--(-1/2-0.015,-1/2-0.04);
\node at (-0.6,-0.65) {\small $m$};
\node at (-0.45,-0.75) {\small $e_1$};
\node at (-0.75,-0.46) {\small $e_2$};
\node at (-0.8,-0.8) {\small $d$};
\end{tikzpicture}
\caption{Illustration of limit taking in Theorem \ref{Thm:main corner (-1/2,-1/2)}}\label{fig:thm corner (-1/2,-1/2)}
\end{figure}
\begin{Thm}[The corner $(\gamma_1,\gamma_2)=(-1/2,-1/2)$]\label{Thm:main corner (-1/2,-1/2)}\
\\
Let $Z_{\gamma_1,\gamma_2}(t)$, $(\gamma_1,\gamma_2)\in \Del$, be defined in (\ref{eq:gen Rosenblatt proc}) with $A=A(\gamma_1,\gamma_2)$ in (\ref{eq:A}). Suppose that $\gamma_1\ge \gamma_2$.
If $(\gamma_1,\gamma_2)\rightarrow (-1/2,-1/2)$ in such a  way that
\begin{equation}\label{eq:rho limit corner 2}
\frac{\gamma_1+1/2}{\gamma_2+1/2}\rightarrow \rho \in [0,1],
\end{equation}
then
\begin{align}\label{eq:Y_rho}
&Z_{\gamma_1,\gamma_2}(t)\Rightarrow  Y_\rho(t)\notag\\ &=t \cdot \left[\frac{(\rho+1)^{-1}+(2\sqrt{\rho})^{-1}}{\sqrt{2(\rho+1)^{-2}+(2\rho)^{-1}}} \cdot X_1+ \frac{(\rho+1)^{-1}-(2\sqrt{\rho})^{-1}}{\sqrt{2(\rho+1)^{-2}+(2\rho)^{-1}}}\cdot X_2\right],
\end{align}
where $X_1$ and $X_2$ two independent standardized  chi-squared random variables with one degree of freedom (with mean $0$ and variance $1$). The case $\rho=0$ is understood as the limit as $\rho\rightarrow 0$.
\end{Thm}

\begin{Rem}
Since by (\ref{eq:hurst}), the self-similarity parameter $H$ equals $\gamma_1+\gamma_2+2$, we get that $H$ tends to $1$ as $(\gamma_1,\gamma_2)\rightarrow (-1/2,-1/2)$. It is known (see e.g., Theorem 3.1.1 of \citet{embrechts:maejima:2002:selfsimilar}) that the only self-similar finite-variance processes with stationary increments having  $H=1$  are degenerate processes. We see this  in Theorem \ref{Thm:main corner (-1/2,-1/2)}, where the limit is a random variable multiplied by $t$.
\end{Rem}

\begin{Rem}
In Theorem \ref{Thm:main corner (-1/2,-1/2)}, if $\rho=1$, $Y_\rho(t)$ reduces to $tX_1$, where $X_1$ is a standardized chi-squared random variable with one degree of freedom. Consider now the standardized Rosenblatt process $Z_\gamma(t)$ in (\ref{eq:Rosenblatt proc}). In this case, $\gamma_1=\gamma_2=\gamma$ and thus $\rho=1$, which corresponds to the middle line $m$ in Figure \ref{fig:thm corner (-1/2,-1/2)}. From Theorem \ref{Thm:main corner (-1/2,-1/2)}, we conclude that if $\gamma\rightarrow -1/2$, then the limit is $tX_1$. This is consistent with a previous result of  \citet{veillette:taqqu:2013:properties}, that the limit  is a standardized chi-squared random variable when $t=1$.
\end{Rem}
\begin{Rem}
If $\rho=0$, $Y_\rho(t)= \frac{t}{\sqrt{2}}(X_1-X_2)$, which has the same distribution as $t \left(W B\right)$, where $W$ and $B$ are two independent standard normal random variables (see (\ref{eq:independent normal}) below). This
is consistent with Theorem \ref{Thm:main edge}, where on the edge $e_1$ the limit is $W B_{\gamma+3/2}$. This tends, as $\gamma\rightarrow -1/2$, to $W\cdot B_1(t)=W\cdot B \cdot t=t(WB)$, where $B$ is a standard Gaussian random variable.  
\end{Rem}

\begin{Rem}
Theorems \ref{Thm:main diag} to \ref{Thm:main corner (-1/2,-1/2)} are consistent with Theorem 3.1 of
\citet{nourdin:poly:2012:convergence}, stating that the limit of a double Wiener-It\^o integral can only be a linear combination of a normal and an independent double Wiener-It\^o integral.
\end{Rem}

\begin{Rem}
Theorem \ref{Thm:main corner (-1/2,-1)} and \ref{Thm:main corner (-1/2,-1/2)} concern the limit behavior of $Z_{\gamma_1,\gamma_2}(t)$ as $(\gamma_1,\gamma_2)$ approaches the corners along some straight-line direction. What happens if one does not approach the corners following a straight-line direction? Then, there will be no convergence.  To see this, consider the case of Theorem \ref{Thm:main corner (-1/2,-1)} (a similar argument can be made  for Theorem \ref{Thm:main corner (-1/2,-1/2)}). Let 
\[
\rho(\gamma_1,\gamma_2)=\frac{\gamma_1+\gamma_2+3/2}{\gamma_2+1}\in (0,1)
\]
parameterize the straight-line direction. Suppose that
 $\rho(\gamma_1,\gamma_2)$ does not converge as $(\gamma_1,\gamma_2)$ approaches the  corner $(-\frac{1}{2},-1)$. Then there are two subsequences of $(\gamma_1,\gamma_2)$, such that $\rho(\gamma_1,\gamma_2)$ of the first subsequence converges to $\rho_1$  and $\rho(\gamma_1,\gamma_2)$ of the second subsequence converges to $\rho_2$, with $\rho_1\neq \rho_2$. By  Theorem \ref{Thm:main corner (-1/2,-1)},  the corresponding processes $Z_{\gamma_1,\gamma_2}(t)$  converge to two different limits.  Therefore, the original process $Z_{\gamma_1,\gamma_2}(t)$  does not converge if $(\gamma_1,\gamma_2)$ does not follow a straight-line direction.
\end{Rem}

\section{Proof of the main theorems}\label{sec:proof}

Since we will use a method of moments,  we state first a cumulant formula for a linear combination of $Z_{\gamma_1,\gamma_2}(t)$ at finite time points.
We let $\kappa_m(\cdot)$  denote the $m$-th cumulant. In the following proposition, the constant $A$ in (\ref{eq:gen Rosenblatt proc}) is arbitrary. 
\begin{Pro}\label{Pro:joint cum}
The $m$-th cumulant ($m\ge 2$) of $\sum_{i=1}^n c_i Z_{\gamma_1,\gamma_2}(t_i)$, $c_i\in \mathbb{R}$, $t_i\in [0,\infty)$, equals
\begin{equation}\label{eq:joint cum gen Ronsen}
\kappa_m\left(\sum_{i=1}^n c_i Z_{\gamma_1,\gamma_2}(t_i)\right)=\frac{1}{2}(m-1)!A^m  C_m(\gamma_1,\gamma_2;\mathbf{t},\mathbf{c}),
\end{equation}
where
\begin{align}
C_m(\gamma_1,\gamma_2;\mathbf{t},\mathbf{c})=\sum_{\sigma \in \{1,2\}^m}&\sum_{i_1,\ldots,i_m=1}^n  c_{i_1}\ldots c_{i_m} \int_0^{t_{i_1}}ds_1\ldots \int_0^{t_{i_m}}ds_m  \notag\\ \prod_{j=1}^m \notag
 \Big[ \notag 
&
(s_j-s_{j-1})_+^{\gamma_{\sigma_j}+\gamma_{\sigma_{j-1}'}+1}\B(\gamma_{\sigma_{j-1}'}+1,-\gamma_{\sigma_j}-\gamma_{\sigma_{j-1}'}-1)   \notag\\  +&(s_{j-1}-s_j)_+^{\gamma_{\sigma_j}+\gamma_{\sigma_{j-1}'}+1} \B(\gamma_{\sigma_{j}}+1,-\gamma_{\sigma_j}-\gamma_{\sigma_{j-1}'}-1) 
\Big] ,\label{eq:C(t1,...,tk)}
\end{align}
where 
\begin{equation}\label{eq:beta fun}
\B(x,y)=\int_0^1 u^{x-1}(1-u)^{y-1}du =\int_0^\infty w^{x-1}(1+w)^{-x-y} dw,~~ x,y>0,
\end{equation}
is the beta function,  the sum runs over $\sigma=(\sigma_1,\ldots,\sigma_m)$ with $\sigma_i=1$ or $2$, and $\sigma'$ is the complement of $\sigma$, namely, $\sigma'_i=1$ if $\sigma_i=2$ and $\sigma_i'=2$ if $\sigma_i=1$, $i=1,\ldots,m$. Moreover $\sigma'_0=\sigma'_m$ and $s_0=s_m$, $i=1,\ldots,m$.
\end{Pro}

Proposition \ref{Pro:joint cum} is an extension of Theorem 2.1 of \citet{bai:taqqu:2013:6-structure}.
We shall use the following cumulant formula for a double Wiener-It\^o integral (see, e.g., (8.4.3) of
\citet{nourdin:peccati:2012:normal}):
\begin{Lem}\label{Lem:double integral cumulant}
If $f$ is a symmetric function in $L^2(\mathbb{R}^2)$, then the $m$-th cumulant of the double Wiener-It\^o integral $X=\int_{\mathbb{R}^2}' f(y_1,y_2)B(dy_1) B(dy_2)$ is given by the following circular integral:
\[
\kappa_m(X)=2^{m-1}(m-1)!  \int_{\mathbb{R}^m} f(y_1,y_2)f(y_2,y_3)\ldots f(y_{m-1},y_m)f(y_m,y_1) dy_1\ldots dy_m.
\]  
\end{Lem}

%\begin{Lem}\label{Lem:useful 2}
%For $a,b>-1$  and $x<y$,  
%\begin{equation}\label{eq:useful 2}
%\int_x^y (u-x)^{a}(y-u)^b  du =(y-x)^{a+b+1} \B(a+1,b+1).
%\end{equation}
%\end{Lem}
%\begin{proof}
%\begin{align*}
%\int_x^y (u-x)^{a}(y-u)^b du& = (y-x)^{a+b+1} \int_x^y \left(\frac{u-x}{y-x}\right)^a \left(\frac{y-u}{y-x}\right)^b d\left(\frac{u}{y-x}\right)
%\\ &=(y-x)^{a+b+1} \int_0^1 w^a (1-w)^b dw.
%\end{align*}
%\end{proof}

\begin{proof}[Proof of Proposition \ref{Pro:joint cum}]

Set 
\[g(x,y)=\frac{A}{2}(x_+^{\gamma_1}y_+^{\gamma_2}+x_+^{\gamma_2}y_+^{\gamma_1}).\] Let 
\[h_t(x,y)=\int_0^t g(s-x,s-y) ds,\] 
and observe that $h_t$ is symmetric. So using the linearity of the Wiener-It\^o integral and Lemma \ref{Lem:double integral cumulant}, we have

\begin{align*}
\kappa_m\left(\sum_{i=1}^n  c_i Z_{\gamma_1,\gamma_2}(t_i)\right)=&\kappa_m\left(\int_{\mathbb{R}^2}' \sum_{i=1}^n c_i h_{t_i}(x_1,x_2)    B(dx_1)B(dx_2) \right)\notag\\
= &2^{m-1}(m-1)! \int_{\mathbb{R}^m}  d\mathbf{x}\prod_{j=1}^m \left[ \sum_{i=1}^n c_i h_{t_i}(x_j,x_{j+1})\right]\notag\\
=&2^{m-1}(m-1)!\sum_{i_1,\ldots,i_m=1}^n  c_{i_1}\ldots c_{i_m} \int_{\mathbb{R}^m} d\mathbf{x} \prod_{j=1}^m \int_0^{t_{i_j}}g(s_j-x_j,s_j-x_{j+1})ds_j,
\end{align*}
and hence
\begin{align}
&\kappa_m\left(\sum_{i=1}^n c_i Z_{\gamma_1,\gamma_2}(t_i)\right)
=\frac{1}{2}(m-1)!  A^m \sum_{i_1,\ldots,i_m=1}^n  c_{i_1}\ldots c_{i_m}  \notag\\ & \times \int_0^{t_{i_1}} ds_1\ldots \int_0^{t_{i_m}}ds_m \Big(~ \int_{\mathbb{R}^m} \prod_{j=1}^m [(s_{j}-x_{j})_+^{\gamma_1} (s_{j}-x_{j+1})_+^{\gamma_2}+(s_{j}-x_{j})_+^{\gamma_2} (s_{j}-x_{j+1})_+^{\gamma_1}]d\mathbf{x}~\Big),\label{eq:cumulant intermediate}
\end{align}
where we  view the index $j$ as modulo $m$, e.g., $x_{m+1}=x_1$.

Then using the notation in the statement of Proposition \ref{Pro:joint cum}, one has
\begin{align*}
I:=&\int_{\mathbb{R}^m} \prod_{j=1}^m \left[(s_{j}-x_{j})_+^{\gamma_1} (s_{j}-x_{j+1})_+^{\gamma_2}+(s_{j}-x_{j})_+^{\gamma_2} (s_{j}-x_{j+1})_+^{\gamma_1}\right]d\mathbf{x}\notag
\\= & \sum_{\sigma\in \{1,2\}^m}  \int_{\mathbb{R}^m}  \prod_{j=1}^{m}(s_{j}-x_j)_+^{\gamma_{\sigma_j}}(s_{j}-x_{j+1})_+^{\gamma_{\sigma_{j}'}} d\mathbf{x}\notag\\
=&\sum_{\sigma\in \{1,2\}^m} \int_{\mathbb{R}^m}  \prod_{j=1}^{m} (s_{j}-x_j)_+^{\gamma_{\sigma_j}}(s_{j-1}-x_{j})_+^{\gamma_{\sigma_{j-1}'}} d\mathbf{x},\notag\\
\end{align*}
and thus
\begin{align}
I=&\sum_{\sigma \in \{1,2\}^m}\prod_{j=1}^m 
 \Big[  
(s_j-s_{j-1})_+^{\gamma_{\sigma_j}+\gamma_{\sigma_{j-1}'}+1}\B(\gamma_{\sigma_{j-1}'}+1,-\gamma_{\sigma_j}-\gamma_{\sigma_{j-1}'}-1) \notag 
\\ &\qquad\qquad~~+(s_{j-1}-s_j)_+^{\gamma_{\sigma_j}+\gamma_{\sigma_{j-1}'}+1} \B(\gamma_{\sigma_{j}}+1,-\gamma_{\sigma_j}-\gamma_{\sigma_{j-1}'}-1) 
\Big],\label{eq:cum last}
\end{align}
where we have used the following relation valid for $a,b\in (-1,-1/2)$:
\begin{equation}\label{eq:int formula}
\int_{\mathbb{R}} (s_1-u)^a_+ (s_2-u)^b_+ du= (s_2-s_1)_+^{a+b+1} B(a+1,-a-b-1) + (s_1-s_2)_+^{a+b+1}\B(b+1,-a-b-1). 
\end{equation}
(See Lemma 3.2 of \citet{bai:taqqu:2013:6-structure}.) Substituting (\ref{eq:cum last}) into (\ref{eq:cumulant intermediate}), equation (\ref{eq:joint cum gen Ronsen}) is obtained. 
\end{proof}

Note that $\E Z_{\gamma_1,\gamma_2}(1)=0$ by the property of Wiener-It\^o integral, and hence the second and the third moments coincide with the second and the third cumulants.
As  two special cases of Proposition \ref{Pro:joint cum}, one has the following explicit formulas for the second and the third moment  of the generalized Rosenblatt distribution (\citet{bai:taqqu:2013:6-structure}, Theorem 2.1):

The second moment of $Z_{\gamma_1,\gamma_2}(1)$ is
\begin{align}
\mu_2(\gamma_1,\gamma_2)  =&  \frac{A^2}{(\gamma_1+\gamma_2+2)(2(\gamma_1+\gamma_2)+3)}\notag\\& \times
\Big[\B(\gamma_{1}+1,-\gamma_{1}-\gamma_{2}-1) \B(\gamma_{2}+1,-\gamma_{1}-\gamma_{2}-1)\notag\\&\quad+ \B(\gamma_{1}+1,-2\gamma_{1}-1) \B(\gamma_{2}+1,-2\gamma_{2}-1)\Big],\label{eq:mu_2}
\end{align}
The third moment of $Z_{\gamma_1,\gamma_2}(1)$ is
\begin{align}
\mu_3(\gamma_1,\gamma_2)=& \frac{2A^3}{(\gamma_1+\gamma_2+2)(3(\gamma_1+\gamma_2)+5)}~\notag \\&\times\Big[ \sum_{\sigma\in \{1,2\}^3}
B(\gamma_{\sigma_1}+1,-\gamma_{\sigma_1}-\gamma_{\sigma_3'}-1)
 B(\gamma_{\sigma_1'}+1,-\gamma_{\sigma_1'}-\gamma_{\sigma_2}-1)
\nonumber \\& 
\quad\times    B(\gamma_{\sigma_2'}+1,-\gamma_{\sigma_2'}-\gamma_{\sigma_3}-1)B(\gamma_{\sigma_1'}+\gamma_{\sigma_2}+2,\gamma_{\sigma_2'}+\gamma_{\sigma_3}+2)\Big]\label{eq:mu_3}.
\end{align}
%Figure \ref{fig:contour} displays a contour plot of the third moment $\mu_3(\gamma_1,\gamma_2)$ in (\ref{eq:mu_3}). 

To standardize $Z_{\gamma_1,\gamma_2}(t)$, we set $\mu_2(\gamma_1,\gamma_2)=1$. By (\ref{eq:mu_2}), this determines the constant $A$ as:
\begin{align}
A(\gamma_1,\gamma_2)=& \Big[(\gamma_1+\gamma_2+2)(2(\gamma_1+\gamma_2)+3)\Big]^{1/2}~\notag\\& \times \Big[\B(\gamma_{1}+1,-\gamma_{1}-\gamma_{2}-1) \B(\gamma_{2}+1,-\gamma_{1}-\gamma_{2}-1)\notag\\
&\quad~~+\B(\gamma_{1}+1,-2\gamma_{1}-1) \B(\gamma_{2}+1,-2\gamma_{2}-1)\Big]^{-1/2}.\label{eq:A}
\end{align}
 
%\begin{figure}
%\centering
%\includegraphics[scale=0.9]{M3C}
%\caption{Contour plot of $\mu_3(\gamma_1,\gamma_2)$. \\{\small
%Boundaries are given by the lines $\gamma_1=-1/2$, $\gamma_2=-1/2$ and $\gamma_1+\gamma_2=-3/2$. The third moment $\mu_3$ vanishes on the diagonal $\gamma_1+\gamma_2=-3/2$. }}\label{fig:contour}
%\end{figure}

\subsection{Proof of Theorem \ref{Thm:main diag}}

We will use a result for bounding integral of powers of linear functions in Euclidean space. First some notation. Let $L_1(\mathbf{s})=\langle\mathbf{w}_1 , \mathbf{s}\rangle,\ldots,L_m(\mathbf{s})=\langle\mathbf{w}_m , \mathbf{s}\rangle$ be linear functions on $\mathbb{R}^n$, where $\langle \cdot, \cdot \rangle$ denotes the Euclidean inner product.  Let
\[
P(\mathbf{s})=\prod_{j=1}^m|L_j(\mathbf{s})|^{\alpha_j}.
\]
Set $T=\{\mathbf{w}_1,\ldots,\mathbf{w}_m\}$. For any nonempty $W\subset T$, define 
\begin{equation}\label{eq:S(W)}
S(W)=T\cap \mathrm{span}\{W\},
\end{equation}
where $\mathrm{span}\{W\}$ denotes linear subspace spanned by $W$, and define the quantity
\[
d(P,W)=|W|+ \sum_{j:\mathbf{w}_j\in S(W)} \alpha_j,
\]
where $|W|$ is the cardinality of the set $W$. 
Then we have the following so-called power counting lemma:
\begin{Lem}[Theorem 3.1 of \citet{fox:taqqu:1987:central}
]\label{Lem:power counting}
Suppose that  
\begin{equation}\label{eq:d(P,W)}
d(P,W)> 0.
\end{equation}
 for any  $W\subset T$ which  consists of linearly independent $\mathbf{w}_j$'s\footnote{Theorem 3.1 of \citet{fox:taqqu:1987:central} states that it is enough to consider $W\subset T$ consisting of linearly independent $\mathbf{w}_j$'s with \emph{negative} exponent $\alpha_j$'s. This is because the non-negative exponents $\alpha_j$ cannot make the integral $\int_{[0,1]^n} P(\mathbf{s}) d\mathbf{s}$ blow up.}. Then
\[
\int_{[0,1]^n} P(\mathbf{s}) d\mathbf{s}<\infty. 
\]
\end{Lem}

\begin{Lem}\label{Lem:bound circular}
The function  
\begin{equation}\label{eq:circular int}
f(\alpha_1,\ldots,\alpha_m):= \int_{[0,1]^m} |s_1-s_m|^{\alpha_1} |s_2-s_1|^{\alpha_2}\ldots |s_m-s_{m-1}|^{\alpha_m} d\mathbf{s}
\end{equation}
is finite and continuous on the domain 
\begin{equation}\label{eq:domain D}
D=\left\{(\alpha_1,\ldots,\alpha_m): ~\alpha_i >-1,~\sum_{i=1}^m \alpha_i+m >1\right\}.
\end{equation}
\end{Lem}
\begin{proof}
We first show that $f(\alpha_1,\ldots,\alpha_m)<\infty$ on $D$ using Lemma \ref{Lem:power counting}. Following the notation introduced for the lemma, we have $L_1(\mathbf{s})=s_1-s_m$, $L_2(\mathbf{s})=s_2-s_1$,\ldots, $L_m(\mathbf{s})=s_m-s_{m-1}$, and hence $\mathbf{w}_1=(1,0,\ldots,0,-1)$, $\mathbf{w}_2=(-1,1,0,\ldots,0)$, \ldots, $\mathbf{w}_m=(0,\ldots,0,-1,1)$ and $T=\{\mathbf{w}_1,\ldots,\mathbf{w}_m\}$.

It is easy to see that  a subset $W\subset T$ consists of linearly independent $\mathbf{w}_j$'s if and only if $|W|\le m-1$. When $|W|\le m-2$, the set $S(W)$ defined in (\ref{eq:S(W)}) is equal to $W$. 
The condition (\ref{eq:d(P,W)}) is satisfied in this case because each $\alpha_j>-1$ and hence
\[
D(P,W)=|W|+ \sum_{j:\mathbf{w}_j\in S(W)} \alpha_j~>~ |W|+ \sum_{j:\mathbf{w}_j\in W} (-1) =|W|-|W|=0.
\] 
When $|W|=m-1$, one has $\mathrm{span}(W)=T$, and hence $S(W)=T$. Thus the condition (\ref{eq:d(P,W)}) in this case becomes 
\[
D(P,W)=m-1+\sum_{i=1}^m \alpha_i>0,
\]
which is satisfied in view of (\ref{eq:domain D}). Hence the integral $f(\alpha_1,\ldots,\alpha_m)$ in  (\ref{eq:circular int}) is finite by Lemma \ref{Lem:power counting}.

To verify the continuity of $f(\alpha_1,\ldots,\alpha_m)$, suppose that as $n\rightarrow\infty$, $\boldsymbol{\alpha}_n\rightarrow \boldsymbol{\alpha}:=(\alpha_1,\ldots,\alpha_m)$. Then for large $n$, $\boldsymbol{\alpha}_n\ge \boldsymbol{\alpha}_\epsilon:=(\alpha_1-\epsilon,\ldots,\alpha_m-\epsilon)$, where the small $\epsilon$ is chosen such that $\boldsymbol{\alpha}_\epsilon\in D$.
Denote  the integrand in (\ref{eq:circular int}) by $I(\mathbf{s};\boldsymbol{\alpha})$, and recall that $I(\mathbf{s};\boldsymbol{\alpha})$ is decreasing in every component of $\boldsymbol{\alpha}$. Hence when $n$ is large,
$I(\mathbf{s};\boldsymbol{\alpha}_n)\le I(\mathbf{s};\boldsymbol{\alpha}_\epsilon)$.
Since $I(\mathbf{s};\boldsymbol{\alpha}_\epsilon)$ is integrable, we can  apply the Dominated Convergence Theorem to obtain the convergence $f(\boldsymbol{\alpha}_n)\rightarrow f(\boldsymbol{\alpha})$ as $n\rightarrow\infty$, proving the continuity.
\end{proof}

In the following corollary, the exponents are supposed to be away from the boundary of the set $D$ defined in (\ref{eq:domain D}).
\begin{Cor}\label{Cor:bound circular}
Let $C_1,C_2$ be two fixed constants such that $C_1>-1$ and $C_2>1$. 
Then the function $f(\alpha_1,\ldots,\alpha_m)$ defined in (\ref{eq:circular int})
is bounded on the domain
\[
D(C_1,C_2)=\left\{(\alpha_1,\ldots,\alpha_m): ~\alpha_i \ge C_1,~\sum_{i=1}^m \alpha_i+m \ge C_2\right\}.
\]
\end{Cor}
\begin{proof}
Let $M$ be a large positive constant. Define
\begin{align*}
D_M(C_1,C_2)&= D(C_1,C_2)\cap (-\infty, M]^m\\
&=\left\{(\alpha_1,\ldots,\alpha_m): ~C_1\le \alpha_i \le M,~\sum_{i=1}^m \alpha_i+m \ge C_2\right\}.
\end{align*}
Since $D_M(C_1,C_2)$ is a compact subset of $D$ in (\ref{eq:domain D}), and $f(\alpha_1,\ldots,\alpha_m)$ is continuous on $D$ by Lemma \ref{Lem:bound circular}, we deduce that $f$ is bounded on $D_M(C_1,C_2)$. The boundedness on $D(C_1,C_2)$ follows since $f$ decreases  when any $\alpha_i$ increases.
\end{proof}

\begin{Lem}\label{Lem:bound A}
Let $A(\gamma_1,\gamma_2)$ be as in (\ref{eq:A}), where $(\gamma_1,\gamma_2)\in \Del$ which is defined in (\ref{eq:region}). Then there exits a constant $C>0$ independent of $\gamma_1$ and $\gamma_2$ such that
\[
|A(\gamma_1,\gamma_2)|\le C[2(\gamma_1+\gamma_2)+3]^{1/2}.
\]
\end{Lem}
\begin{proof}
This is immediate by noting that the beta function $\B(x,y)$ defined in (\ref{eq:beta fun}) is decreasing in $x$ and in $y$. 
Since in addition $\Del$ is a bounded region, the beta functions in (\ref{eq:A}) are bounded from below, and hence the factor with negative power $-1/2$ in (\ref{eq:A}) is bounded from above.
\end{proof}

The following  hypercontractivity inequality for multiple Wiener-It\^o integral (see, e.g., Corollary 5.6 of \citet{major:1981:multiple} or Theorem 2.7.2 of \citet{nourdin:peccati:2012:normal}) is useful: 
\begin{Lem}\label{Lem:hypercontract}
For any $m\in \mathbb{Z}_+$, there exists a constant $C_m>0$, such that
\[
\E |I_k(f)|^{2m} \le C_m \left(\E |I_k(f)|^2\right)^m, \quad\text{ for all $f\in L^2(\mathbb{R}^k)$}.
\]
\end{Lem}

Tightness of standardized $Z_{\gamma_1,\gamma_2}(t)$ in $C[0,1]$ will follow from the following lemma:
\begin{Lem}\label{Lem:tight}
Let $Z_{\gamma_1,\gamma_2}(t)$ be as in (\ref{eq:gen Rosenblatt proc}) with $A$ as in (\ref{eq:A}) and $(\gamma_1,\gamma_2)$ in the region $\Del$ defined in (\ref{eq:region}). Then there exists a constant $C>0$ which does not depend on $\gamma_1,\gamma_2$, such that
for all $0\le s\le t\le 1$,
\[
\E |Z_{\gamma_1,\gamma_2}(t)-Z_{\gamma_1,\gamma_2}(s)|^4\le C (t-s)^{2},
\]
which implies that the law of  $\{Z_{\gamma_1,\gamma_2}(t): ~(\gamma_1,\gamma_2)\in \Del\}$ is tight in $C[0,1]$. 
\end{Lem}
\begin{proof}
Using Lemma \ref{Lem:hypercontract}, self-similarity and stationary-increment property of $Z_{\gamma_1,\gamma_2}(t)$, one has
\begin{align*}
\E |Z_{\gamma_1,\gamma_2}(t)-Z_{\gamma_1,\gamma_2}(s)|^4&\le C_2 \left(\E |Z_{\gamma_1,\gamma_2}(t)-Z_{\gamma_1,\gamma_2}(s)|^2\right)^2\\
&= C_2 (t-s)^{4H}\le C_2 (t-s)^{2},
\end{align*}
where $H:=\gamma_1+\gamma_2+2\ge 1/2$ and $0\le t-s \le 1$. So $Z_{\gamma_1,\gamma_2}(t)$ by Kolmogorov's criterion admits a continuous version. Tightness  follows from, e.g., \citet{prokhorov:1956:convergence} Lemma 2.2.
\end{proof}

We now prove Theorem  \ref{Thm:main diag}. By Lemma \ref{Lem:tight},  tightness in $C[0,1]$ holds. We are left to show convergence of finite-dimensional distributions ($\ConvFDD$).  From here on, we let $C$ and $c$ denote constants whose values can change from line to line.
\begin{proof}[Proof  of $\ConvFDD$  Theorem \ref{Thm:main diag}]
Due to self-similarity and stationary increments, the covariance  of the standardized $Z_{\gamma_1,\gamma_2}(t)$ is 
\[
\E Z_{\gamma_1,\gamma_2}(s)Z_{\gamma_1,\gamma_2}(t)=\frac{1}{2}\left(s^{2\gamma_1+2\gamma_2+4}+t^{2\gamma_1+2\gamma_2+4}-|s-t|^{2\gamma_1+2\gamma_2+4}\right), \quad t,s\ge 0,
\] 
which converges to the Brownian motion covariance $\E B(s)B(t)=s\wedge t=\frac{1}{2}(s+t-|s-t|)$ as  $\gamma_1+\gamma_2\rightarrow -3/2$. By using the method of moments,  it is sufficient to show that
\begin{equation}\label{eq:cum conv zero}
\kappa_m\left(\sum_{i=1}^n c_i Z_{\gamma_1,\gamma_2}(t_i) \right)\rightarrow 0, \quad \text{ m}\ge 3.
\end{equation}

As $\gamma_1+\gamma_2\rightarrow -3/2$, the factor $A(\gamma_1,\gamma_2)$ in (\ref{eq:A}) converges to zero  by Lemma \ref{Lem:bound A}. It is therefore sufficient to show that for $m\ge  3$, and $\gamma_1,\gamma_2>-1+\epsilon$, the factor $C_m(\gamma_1,\gamma_2;\mathbf{t},\mathbf{c})$ in (\ref{eq:C(t1,...,tk)}) is bounded.

Under the constraints  $\gamma_1+\gamma_2\ge -3/2$ and $\gamma_1,\gamma_2>-1+\epsilon$ (or equivalently  $\gamma_1,\gamma_2<-1/2-\epsilon$), the factors  $\B(\gamma_{\sigma_{j-1}'}+1,-\gamma_{\sigma_j}-\gamma_{\sigma_{j-1}'}-1)$ and $\B(\gamma_{\sigma_j}+1,-\gamma_{\sigma_j}-\gamma_{\sigma_{j-1}'}-1)$   are bounded by a constant $C>0$ for any $\sigma$ and $j$. This is because the beta function $\B(x,y)$ defined in (\ref{eq:beta fun}) is bounded if both $x$ and $y$ stay away from a neighborhood of $0$.
 Choosing $T\ge \max (t_1,\ldots,t_n)$, one then has
\begin{align*}
|C_m(\gamma_1,\gamma_2;\mathbf{t},\mathbf{c})| \le&   C  \sum_{\sigma \in \{1,2\}^m}\int_{[0,T]^m} d\mathbf{s} \prod_{j=1}^{m} |s_j-s_{j-1}|^{\gamma_{\sigma_j}+\gamma_{\sigma_{j-1}'}+1}  \\ 
\le&  　C \sum_{\sigma \in \{1,2\}^m} \int_{[0,1]^m} d\mathbf{s} \prod_{j=1}^{m} |s_j-s_{j-1}|^{\gamma_{\sigma_j}+\gamma_{\sigma_{j-1}'}+1},
\end{align*}
where the last constant $C$ depends on $T$, $m$ and $\epsilon$.

We now want to apply  Corollary \ref{Cor:bound circular} to establish the boundedness of each of the term in the preceding sum.    Using the notation in Lemma \ref{Lem:bound circular}, we set 
\[
\alpha_j=\gamma_{\sigma_j}+\gamma_{\sigma_{j-1}'}+1.
\] 
Recall that $\gamma_{\sigma_j}$ and $\gamma_{\sigma_{j-1}'}$ are either $\gamma_1$ or $\gamma_2$ and $\gamma_{\sigma_j}+\gamma_{\sigma_j'}=\gamma_1+\gamma_2$. Now
since  $\gamma_1+\gamma_2\ge -3/2$ and $\gamma_j\ge -1+\epsilon$,  we have
\begin{align*}
\alpha_j\ge   
\begin{cases} 
2\gamma_j+1\ge -1+2\epsilon, &\text{ if } \sigma_{j-1}'=\sigma_{j};\\
\gamma_1+\gamma_2+1\ge -3/2+1=-1/2, &\text{ if } \sigma_{j-1}'\neq \sigma_{j};
\end{cases}
\end{align*}
We get $\alpha_j\ge  C_1:=-1+2\epsilon>-1$.

On the other hand, when $m\ge 3$,
\begin{align*}
\sum_{i=1}^m \alpha_i+m=m(\gamma_1+\gamma_2)+2m\ge m(-3/2) +2m =\frac{m}{2}\ge C_2:=\frac{3}{2} >1.
\end{align*}
So Corollary \ref{Cor:bound circular} can be applied to deduce the boundedness of  $|C_m(\gamma_1,\gamma_2;\mathbf{t},\mathbf{c})|$ when $\gamma_1,\gamma_2\ge -1+\epsilon$, and the proof is thus concluded.
\end{proof}
\begin{Rem}
Theorem \ref{Thm:main diag} involves convergence to a Gaussian process. In this case, 
according to the results of \citet{nualart:peccati:2005:central} and \citet{peccati:tudor:2005:gaussian}, it suffices to show that (\ref{eq:cum conv zero}) holds for $m=4$ and $n=1$. 
Focusing on the fourth cumulant, the covariance structure, and the one-dimensional distribution, however, does not simplify significantly the proof as can be seen by examining the  proof of Theorem \ref{Thm:main diag}. 
\end{Rem}

\subsection{Proof of Theorem \ref{Thm:main edge}}
\begin{Lem}\label{Lem:fractional cov}
Suppose that $\alpha>-1$, then for any $t_1,t_2\in \mathbb{R}$,
\[
\int_0^{t_1}\int_0^{t_2}|x_1-x_2|^{\alpha} dx_1dx_2=\frac{1}{(\alpha+1)(\alpha+2)} \left( |t_1|^{\alpha+2}+ |t_2|^{\alpha+2}-|t_1-t_2|^{\alpha+2} \right).
\]
\end{Lem}
\begin{proof}
Suppose $0<t_1\le t_2$. The other cases  are similar. Then
\begin{align*}
\int_0^{t_1}\int_0^{t_2}|x_1-x_2|^{\alpha} dx_1dx_2
&=\int_0^{t_1} \int_0^{t_1} |x_1-x_2|^{\alpha} dx_1dx_2+ \int_0^{t_1}\int_{t_1}^{t_2} (x_2-x_1)^\alpha dx_2 dx_1\\
&=\frac{2}{(\alpha+1)(\alpha+2)}t_1^{\alpha+2}+\frac{1}{(\alpha+1)(\alpha+2)}[t_2^{\alpha+2}-t_1^{\alpha+2}-(t_2-t_1)^{\alpha+2}]\\
&=\frac{1}{(\alpha+1)(\alpha+2)}\left[ t_1^{\alpha+2}+ t_2^{\alpha+2}-(t_2-t_1)^{\alpha+2} \right].
\end{align*}
\end{proof}

Below the notation $A\sim B$ means asymptotic equivalence, namely, the ratio $A/B$ converges to $1$.
We include first a fact about the asymptotics of the beta function $\B(\cdot,\cdot)$ when one of the exponents approaches the boundary. 
\begin{Lem}\label{Lem:beta asymp}
 Let $0<b_0<b_1<\infty$. Then as $\alpha\rightarrow 0$, we have
\[
\alpha\B(\alpha,\beta)\rightarrow 1
\]
uniformly in $\beta \in [b_0,b_1]$.
Since the beta functions is symmetric, we also have $\alpha\B(\beta,\alpha)\rightarrow 1$ as $\alpha\rightarrow 0$ uniformly in $\beta \in [b_0,b_1]$.
\end{Lem}
\begin{proof}
Assume without loss of generality that $b_0\le 1 \le b_1$.
Fix any small $\epsilon>0$. Then 
\begin{equation}\label{eq:B=I_1+I_2}
 \B(\alpha,\beta)=\int_0^\epsilon x^{\alpha-1}(1-x)^{\beta-1} dx+\int_{\epsilon}^1 x^{\alpha-1}(1-x)^{\beta-1} dx=: I_1(\alpha,\beta;\epsilon)+ I_2(\alpha,\beta;\epsilon).
\end{equation}
For $I_1(\alpha,\beta;\epsilon)$, we have 
\begin{align*}
\alpha^{-1} \epsilon^{\alpha}(1-\epsilon)^{b_1-1}= \int_0^\epsilon x^{\alpha-1} dx (1-\epsilon)^{b_1-1}
\le I_1(\alpha,\beta;\epsilon)\le \int_0^\epsilon x^{\alpha-1} dx  (1-\epsilon)^{b_0-1}= \alpha^{-1} \epsilon^{\alpha} (1-\epsilon)^{b_0-1}.
\end{align*}
This yields that
\begin{align}\label{eq:I_1 est}
(1-\epsilon)^{b_1-1}\le \liminf_{\alpha\rightarrow0, \beta\in [b_0,b_1]} \alpha I_1(\alpha,\beta,\epsilon)\le \limsup_{\alpha\rightarrow0, \beta\in [b_0,b_1]} \alpha I_1(\alpha,\beta,\epsilon) \le    (1-\epsilon)^{b_0-1}.
\end{align}
For $I_2(\alpha,\beta;\epsilon)$, it is uniformly bounded with respect to $\alpha\le 1$ and $\beta$ as follows: 
\begin{equation}\label{eq:I_2 est}
I_2(\alpha,\beta;\epsilon)\le \epsilon^{\alpha-1} \int_\epsilon^1 (1-x)^{\beta-1}dx =\epsilon^{\alpha-1}  \beta^{-1} (1-\epsilon)^\beta \le \epsilon^{-1} b_0^{-1} (1-\epsilon)^{b_0}.
\end{equation}
Combining (\ref{eq:B=I_1+I_2}), (\ref{eq:I_1 est}) and (\ref{eq:I_2 est}), we get
\[
(1-\epsilon)^{b_1-1}\le \liminf_{\alpha\rightarrow0, \beta\in [b_0,b_1]} \alpha \B(\alpha,\beta)\le \limsup_{\alpha\rightarrow0, \beta\in [b_0,b_1]} \alpha\B(\alpha,\beta) \le    (1-\epsilon)^{b_0-1}.
\]
Since $\epsilon$ is arbitrary, we get that $\alpha \B(\alpha,\beta)\rightarrow 1$ as $\alpha\rightarrow 0$.
\end{proof}
The limit $\alpha \B(\alpha,\beta)\rightarrow 1$ as $\alpha\rightarrow 0$ will be used extensively, mostly in the form 
\[\B(\alpha,\beta)\sim \alpha^{-1}\rightarrow \infty.\]

\begin{Lem}\label{Lem:Z_gamma limit cumulant}
Let $W B_{\gamma+3/2}(t)$ be the process given as  Theorem \ref{Thm:main edge}. We also include the case $\gamma=-1$ where $B_{\gamma+3/2}(t)=B_{1/2}(t)$ is Brownian motion. Then the $m$-th cumulant of the linear combination of $ WB_{\gamma+3/2}(t)$ at different time points is given by
\begin{align}\label{eq:limit cumulant}
\kappa_m\left(\sum_{i=1}^n c_i W B_{\gamma+3/2}(t_i) \right)
=
 (m-1)! \left[\sum_{i_1,i_2=1}^n \frac{c_{i_1}c_{i_2}}{2} \left(|t_{i_1}|^{2\gamma+3}+|t_{i_2}|^{2\gamma+3}- |t_{i_1}-t_{i_2}|^{2\gamma+3} \right) \right]^{m/2} 
\end{align}
if $m$ is even, and $0$ if $m$ is odd.
\end{Lem}
\begin{proof}
\begin{align*}
\sum_{i=1}^n c_i  WB_{\gamma+3/2}(t_i)=W\sum_{i=1}^nc_iB_{\gamma+3/2}(t_i)=\sigma W Z, 
\end{align*}
where $Z$ is a standard normal random variable which is independent of $W$, and 
\begin{align*}
\sigma=&\left(\Var\left[\sum_{i=1}^n c_i B_{\gamma+3/2}(t_i) \right]\right)^{1/2}=\left[\E  \sum_{i_1,i_2=1}^n c_{i_1}c_{i_2} B_{\gamma+3/2}(t_{i_1}) B_{\gamma+3/2}(t_{i_2}) \right]^{1/2}\\
=&\left[\sum_{i_1,i_2=1}^n \frac{c_{i_1}c_{i_2}}{2} \left(|t_{i_1}|^{2\gamma+3}+|t_{i_2}|^{2\gamma+3}- |t_{i_1}-t_{i_2}|^{2\gamma+3} \right) \right]^{1/2},
\end{align*}
using the covariance of fractional Brownian motion. 
Then note that 
\begin{equation}\label{eq:independent normal}
 W Z= \frac{1}{2}\left[ \left(\frac{W+ Z}{\sqrt{2}}\right)^2- \left(\frac{W- Z}{\sqrt{2}}\right)^2\right],
\end{equation}
where $Z_1^2:=\left[\frac{W+ Z}{\sqrt{2}}\right]^2$ and $Z_2^2:=\left[\frac{W- Z}{\sqrt{2}}\right]^2$ are two independent $\chi_1^2$ (chi-squared random variables with one degree of freedom). The independence is due to the fact that $Z+W$ and $Z-W$ are uncorrelated. Since the $m$-th cumulant of a $\chi_1^2$ variable is $2^{m-1}(m-1)!$, and using the scaling property and the additive property of cumulant under independence, we have
\begin{align*}
\kappa_m\left(\sigma W  Z \right)&=\left(\frac{\sigma}{2}\right)^m [\kappa_m(Z_1^2)+(-1)^m\kappa_m(Z_2^2)]\\
&=\left(\frac{\sigma}{2}\right)^m [2^{m-1}(m-1)!+(-1)^m2^{m-1}(m-1)!],
\end{align*}
which is equal to $0$ if $m$ is odd, and equal to $\sigma^m (m-1)!$ if $m$ is even, proving (\ref{eq:limit cumulant}).
\end{proof}
\begin{Rem}\label{Rem:analytic}
Starting with the $\chi_1^2$ characteristic function $\phi(t)=(1-2it)^{-1/2}$, it is easy to derive using (\ref{eq:independent normal}) that the characteristic function of the  standard   product-normal distribution $W Z$ is $\varphi(t)=(1+t^2)^{-1/2}$.
\end{Rem}

In view of Lemma \ref{Lem:tight},
we are left to prove the convergence of the finite-dimensional distributions ($\ConvFDD$) in  Theorem \ref{Thm:main edge}. 

\begin{proof}[Proof of $\ConvFDD$ in Theorem \ref{Thm:main edge}]
By the Cram\'er-Wold device, we need to show   as $\gamma_1\rightarrow -1/2$ and $\gamma_2\rightarrow \gamma \in (-1/2,-1)$ that
\[
\sum_{i=1}^n c_i Z_{\gamma_1,\gamma_2}(t_i) \ConvD \sum_{i=1}^n c_i W  B_{\gamma+3/2}(t_i) .
\]
Since $\sum_{i=1}^n c_i W  B_{\gamma+3/2}(t_i)$ has an analytic characteristic function (Remark \ref{Rem:analytic}), its distribution is moment-determinate. And hence we can apply a method of moments here. In fact, by Theorem 3.4 of \citet{nourdin:poly:2012:convergence}, only a finite number of moments are required to prove convergence in distribution.

 The cumulant formula of $\sum_{i=1}^n c_i Z_{\gamma_1,\gamma_2}(t_i)$ is given in  Proposition \ref{Pro:joint cum}, which involves the factors $A(\gamma_1,\gamma_2)$ in (\ref{eq:A}) (recall that $Z_{\gamma_1,\gamma_2}$ is standardized) and $C_m(\gamma_1,\gamma_2;\mathbf{t},\mathbf{c})$ in (\ref{eq:C(t1,...,tk)}). Assume $m\ge 2$ below.

Examining $A(\gamma_1,\gamma_2)$,  by Lemma \ref{Lem:beta asymp}, one can see that as  $\gamma_1\rightarrow -1/2$ and $\gamma_2\rightarrow \gamma$,  
\begin{align*}
A(\gamma_1,\gamma_2)^m &\sim \left[(\gamma+3/2)(2\gamma+2)\right]^{m/2}\big[\B(1/2,-\gamma-1/2)\B(\gamma+1,-\gamma-1/2)\notag\\
&\qquad  \qquad  \qquad  \qquad  \qquad  ~ + \B(1/2,-2\gamma_1-1)\B(\gamma+1,-2\gamma-1) \big]^{-m/2}.
\end{align*}
The first two  and the fourth beta functions are bounded but the third blows up since 
\[\B(1/2,-2\gamma_1-1)\sim (-2\gamma_1-1)^{-1}
\] as $\gamma_1\rightarrow -1/2$ by Lemma \ref{Lem:beta asymp}. Hence as $\gamma_1\rightarrow -1/2$, 
\begin{align}\label{eq:A(gamma)^m}
A(\gamma_1,\gamma_2)^m& \sim \left[(\gamma+3/2)(2\gamma+2)\right]^{m/2}\big[\B(1/2,-2\gamma_1-1)\B(\gamma+1,-2\gamma-1) \big]^{-m/2}\notag
\\& \sim (-2\gamma_1-1)^{m/2} (2\gamma+3)^{m/2}(\gamma+1)^{m/2} \B(\gamma+1,-2\gamma-1)^{-m/2}, 
\end{align}
which converges to zero. 
%So
%\begin{equation}\label{eq:A conv order}
%\text{ $A(\gamma_1,\gamma_2)^m$ converges to $0$ following the order $(-2\gamma_1-1)^{m/2}$。 }
%\end{equation}
%

On the other hand, in the expression of $C_m(\gamma_1,\gamma_2;\mathbf{t},\mathbf{c})$ in (\ref{eq:C(t1,...,tk)}), the only factors diverging to $\infty$ as $\gamma_1\rightarrow -1/2$ and $\gamma_2\rightarrow \gamma$ are $\B(\gamma_{\sigma_{j-1}'}+1,-\gamma_{\sigma_j}-\gamma_{\sigma_{j-1}'}-1)$ and $\B(\gamma_{\sigma_j}+1,-\gamma_{\sigma_j}-\gamma_{\sigma_{j-1}'}-1)$ and only when $\sigma_j=\sigma_{j-1}'=1$, because $-\gamma_{\sigma_j}-\gamma_{\sigma_{j-1}'}-1=-2\gamma_1-1\rightarrow 0$ and hence the beta functions each diverge like $(-2\gamma_1-1)^{-1}$ by Lemma \ref{Lem:beta asymp}. To get the highest order of divergence to $\infty$, one chooses $\sigma\in \{1,2\}^m$ such that $\sigma_j=\sigma_{j-1}'=1$ happens as many times as possible.

\medskip

In the case $m$ is odd,
\[
\underset{\sigma\in \{1,2\}^m}{\mathrm{max}} ~\# \{j: \sigma_j=\sigma_{j-1}'=1, j=1,\ldots,m\}=(m-1)/2,
\]
because if $\sigma_j=\sigma_{j-1}'=1$, then $\sigma_j'=2$, and we therefore cannot have $\sigma_{j+1}=\sigma_{j}'=1$.
So 
\begin{equation}\label{eq:C_m odd div order}
C_m(\gamma_1,\gamma_2;\mathbf{t},\mathbf{c})\sim c B(1/2,-2\gamma_1-1)^{(m-1)/2}\sim c (-2\gamma_1-1)^{-(m-1)/2},
\end{equation}
which diverges to $\infty$
as $\gamma_1\rightarrow -1/2$. By (\ref{eq:A(gamma)^m}) and (\ref{eq:C_m odd div order}),  when $m$ is odd,  
\begin{equation}\label{eq:cum odd conv}
\kappa_m\left(\sum_{i=1}^n c_i Z_{\gamma_1,\gamma_2}(t_i) \right)=\frac{1}{2}(m-1)!A(\gamma_1,\gamma_2)^m  C_m(\gamma_1,\gamma_2;\mathbf{t},\mathbf{c})\sim c (-2\gamma_1-1)^{1/2}\rightarrow 0.
\end{equation}

When $m$ is even, the sequences $\sigma$ for which one has the greatest number of $j$'s such that $\sigma_j=\sigma_{j-1}'=1$ is  
\begin{align}\label{eq:1,2,1,2}
\underset{\sigma\in \{1,2\}^m}{\mathrm{argmax}} ~\# \{j: \sigma_j=\sigma_{j-1}'=1, j=1,\ldots,m\}=
(1,2,1,2,\ldots,1,2) \text{ or } (2,1,2,1,\ldots,2,1),
\end{align}
and one gets maximally $m/2$ number of $j$'s where $\sigma_j=\sigma_{j-1}'=1$. The product of the $m/2$ contributing beta factors diverge like $(-2\gamma_1-1)^{m/2}$. But since the case $m$ even will yield a nonzero limit, we need to keep track of the multiplicative constants.  Because $\sigma=(1,2,1,2\ldots,1,2)$ and $\sigma=(2,1,2,1,\ldots,2,1)$  yield the same term, one has as $\gamma_1\rightarrow -1/2$ and $\gamma_2\rightarrow \gamma$ that 
\begin{align}
C_m(\gamma_1,\gamma_2;\mathbf{t},\mathbf{c})\sim & 2(-2\gamma_1-1)^{-m/2}  \Bigg[ \sum_{i_1,\ldots,i_n=1}^n c_{i_1}\ldots c_{i_m} \B(\gamma+1,-2\gamma-1)^{m/2}
 \notag
\\&\times \int_0^{t_{i_1}} \ldots\int_0^{t_{i_m}} |s_1-s_2|^{2\gamma+1}|s_3-s_4|^{2\gamma+1} \ldots |s_{m-1}-s_m|^{2\gamma+1}   d\mathbf{s}\Bigg]\notag\\
=&2(-2\gamma_1-1)^{-m/2} (2\gamma+3)^{-m/2} (\gamma+1)^{-m/2}  B(\gamma+1,-2\gamma-1)^{m/2}\notag\\& \times \left[\sum_{i_1,i_2=1}^n \frac{c_{i_1}c_{i_2}}{2} \left(|t_{i_1}|^{2\gamma+3}+|t_{i_2}|^{2\gamma+3}- |t_{i_1}-t_{i_2}|^{2\gamma+3} \right) \right]^{m/2},\label{eq:C asymp}
\end{align}
where the asymptotic equivalence $\sim$ in the first line can be justified by the Dominated Convergence Theorem, and the last equality is due to Lemma \ref{Lem:fractional cov}. 
 
Combining (\ref{eq:joint cum gen Ronsen}), (\ref{eq:A(gamma)^m}) and (\ref{eq:C asymp}), one gets  as $\gamma_1\rightarrow -1/2$ and $\gamma_2\rightarrow \gamma$ that for $m$ even,
\begin{align}\label{eq:cum even conv}
\kappa_m\left(\sum_{i=1}^n c_i Z_{\gamma_1,\gamma_2}(t_i)\right)\rightarrow
 (m-1)!  \left[\sum_{i_1,i_2=1}^n \frac{c_{i_1}c_{i_2}}{2} \left(|t_{i_1}|^{2\gamma+3}+|t_{i_2}|^{2\gamma+3}- |t_{i_1}-t_{i_2}|^{2\gamma+3} \right) \right]^{m/2}.
\end{align}
The proof is concluded by comparing (\ref{eq:cum odd conv}) and (\ref{eq:cum even conv}) with Lemma \ref{Lem:Z_gamma limit cumulant}.
\end{proof}

We state a byproduct of the preceding proof which will be used in Section \ref{sec:add 2}.
\begin{Cor}\label{Cor:kappa}
Under the condition and the notation of Theorem \ref{Thm:main edge},  when $m\ge 4$ is even, we have
\[
\kappa_m\left(Z_{\gamma_1,\gamma_2}(1)\right)= (m-1)! + O\left(-\gamma_1-1/2\right).
\]
\end{Cor}
\begin{proof}
We are focusing here on the marginal distribution and hence $t=1$, $c=1$ and $n=1$ in (\ref{eq:C(t1,...,tk)}).
To get the rate of convergence $O(-\gamma_1-1/2)$, we need to expand $C_m(\gamma_1,\gamma_2;1,1)$ to a higher order than (\ref{eq:C asymp}).
Following the preceding proof of Theorem \ref{Thm:main edge}, we need to consider the
$\sigma$'s with the second most  occurrences of $\sigma_{j-1}'=\sigma_{j}=1$. These $\sigma$'s have $\sigma_{j-1}'=\sigma_{j}=1$ occurring $m/2-1$ times instead of $m/2$ times as in (\ref{eq:1,2,1,2}).  Adding this type of $\sigma$'s into (\ref{eq:C asymp}), we have
\[
C_m(\gamma_1,\gamma_2;1,1)=c_{\gamma,m}(-\gamma_1-1/2)^{-m/2}+O\left((-\gamma_1-1/2)^{-m/2+1}\right),
\]
where $c_{\gamma,m}$ is the constant given by (\ref{eq:C asymp}) with $t=1$, $c=1$ and $n=1$.
By Proposition \ref{Pro:joint cum},
\[
\kappa_m\left(Z_{\gamma_1,\gamma_2}(1) \right)=\frac{1}{2}(m-1)!A(\gamma_1,\gamma_2)^m  C_m(\gamma_1,\gamma_2;1,1).
\]
So the conclusion follows in view of the expression $A(\gamma_1,\gamma_2)^m$ in (\ref{eq:A(gamma)^m}). 
\end{proof}

\subsection{Proof of Theorem \ref{Thm:main corner (-1/2,-1)}}

\begin{Lem}\label{Lem:delicate asymp}
Let $t_1,\ldots,t_m>0$, and $m\ge 4$ be an even integer. Consider the function:
\begin{align}\label{eq:f(a,b;t)}
f(a,b;\mathbf{t})=\int_0^{t_1}\ldots\int_0^{t_m}& |x_1-x_m|^a |x_2-x_1|^b |x_3-x_2|^a |x_4-x_3|^b \ldots\\
&\times |x_{m-1}-x_{m-2}|^a |x_{m}-x_{m-1}|^b d\mathbf{x},\notag
\end{align}
where $-1<a,b<0$.
Then as $(a,b)\rightarrow(0,-1)$, we have that
$$
f(a,b;\mathbf{t})\sim  (b+1)^{-m/2} \prod_{i=2,4,...m} \left( t_i+ t_{i-1}-|t_{i}-t_{i-1}| \right).
$$
\end{Lem}
\begin{proof}
First,  assume without loss of generality  that $t_1,\ldots t_m<1$. Otherwise one can scale them by a change of variables. 

We first derive a lower bound for $f(a,b;\mathbf{t})$. Since each $|x_i-x_{i-1}|^{a}\ge 1$, one has by Lemma \ref{Lem:fractional cov} that
\begin{align}\label{eq:lower est}
f(a,b;\mathbf{t})\ge f(0,b;\mathbf{t})&=\prod_{i=2,4,...m} \int_{0}^{t_i} \int_0^{t_{i-1}} |x_i-x_{i-1}|^{b} dx_{i}dx_{i-1} \notag
\\&=(b+1)^{-m/2}(b+2)^{-m/2} \prod_{i=2,4,...m} \left( t_i^{b+2}+ t_{i-1}^{b+2}-|t_{i}-t_{i-1}|^{b+2} \right)\notag \\&
\sim (b+1)^{-m/2} \prod_{i=2,4,...m} \left( t_i+ t_{i-1}-|t_{i}-t_{i-1}| \right) \text{\quad as $b\rightarrow -1$.
}
\end{align}

To get an upper bound for $f(a,b;\mathbf{t})$, we  apply the Cauchy-Schwarz inequality  to break the cyclic structure. In particular in  (\ref{eq:f(a,b;t)}), view $|x_1-x_m|^{a}|x_3-x_2|^{a}$ as the integrand, and treat the other factors as the density of measure.  We have
\begin{align}\label{eq:C-S}
f(a,b;\mathbf{t}) \le \sqrt{f_1(a,b;\mathbf{t})f_2(a,b;\mathbf{t})},
\end{align}
where
\begin{align*}
f_1(a,b;\mathbf{t})=\int_0^{t_1}dx_1\ldots \int_0^{t_m} dx_m &|x_1-x_m|^{2a} |x_2-x_1|^b |x_4-x_3|^b|x_5-x_4|^a \ldots\\
\times& |x_{m-1}-x_{m-2}|^a |x_{m}-x_{m-1}|^b,
\end{align*}
and
\begin{align*}
f_2(a,b;\mathbf{t})=\int_0^{t_1}dx_1\ldots \int_0^{t_m} dx_m &|x_3-x_2|^{2a} |x_2-x_1|^b |x_4-x_3|^b|x_5-x_4|^a \ldots\\
\times & |x_{m-1}-x_{m-2}|^a |x_{m}-x_{m-1}|^b.
\end{align*}
Set 
\[
|x|^a=1+h_a(x).
\]
Then the integrand in $f_1$ can be rewritten as 
\[
[1+h_{2a}(x_1-x_m)] |x_2-x_1|^b |x_4-x_3|^b[1+h_a(x_5-x_4)] \ldots [1+h_{a}(x_{m-1}-x_{m-2})] |x_{m}-x_{m-1}|^b.
\]
Observe that the product of terms involving neither $h_a$ nor $h_{2a}$ equals $f(0,b;\mathbf{t})$. Hence one can write
\begin{align*}
f_1(a,b;\mathbf{t})=f(0,b;\mathbf{t})+R(a,b;\mathbf{t}),
\end{align*}
where  the remainder $R(a,b;\mathbf{t})$ is a sum of terms each  involving at least one $h_{a}$ or $h_{2a}$. We claim that $|R(a,b;\mathbf{t})|=o\left((b+1)^{-m/2}\right)$. Indeed, let $R_1(a,b;\mathbf{t})$ be the term of $R(a,b;\mathbf{t})$  involving only one $h_{2a}$ and no other $h_{a}$.
Using the fact that when $f$ is a non-negative function and $0<x_1,x_2<t$, we have
\[
\int_0^t f(x_2-x_1)dx_2=\int_{-x_1}^{t-x_1} f(x)dx \le \int_{-1}^{1}f(x)dx.
\]
Therefore,
\begin{align}
&|R_1(a,b;\mathbf{t})|\notag\\
=&\int_0^{t_1}dx_1\ldots \int_0^{t_m} dx_m~ h_{2a}(x_1-x_m) |x_2-x_1|^b |x_4-x_3|^b \ldots |x_{m}-x_{m-1}|^b\notag\\
\le &\int_0^{t_1}dx_1\int_0^{t_3}dx_3\ldots \int_0^{t_m} dx_m ~h_{2a}(x_1-x_m) \int_{-1}^1|x_2|^b dx_2~ |x_4-x_3|^b \ldots |x_{m}-x_{m-1}|^b\notag\\
\le& 2(b+1)^{-1}\int_0^{t_3}dx_3\ldots \int_0^{t_m} dx_m \int_{-1}^{1} h_{2a}(x_1) dx_1 ~ |x_4-x_3|^b \ldots |x_{m}-x_{m-1}|^b\notag\\
\le& 2(b+1)^{-1}\int_0^{t_3}dx_3\ldots \int_0^{t_m} dx_m \int_{-1}^{1} (|x_1|^{2a}-1) dx_1 ~ |x_4-x_3|^b \ldots |x_{m}-x_{m-1}|^b\notag \\
= &4[(2a+1)^{-1}-1] (b+1)^{-1} \int_0^{t_3}dx_3\ldots \int_0^{t_m} dx_m ~ |x_4-x_3|^b|x_6-x_5|^b \ldots |x_{m}-x_{m-1}|^b  \notag\\
\le &\ldots \le C [(2a+1)^{-1}-1]  (b+1)^{-m/2}=o(1)(b+1)^{-m/2}. \label{eq:delicate bound}
\end{align}
Similar estimates apply to the other terms of $R(a,b;\mathbf{t})$, which may involve a greater number of $h_a$ or $h_{2a}$, and end up converging faster to zero as $a\rightarrow 0$. Hence
\[
f_1(a,b;\mathbf{t})\le f(0,b;\mathbf{t})+o\left( (b+1)^{-m/2} \right)  
\sim (b+1)^{-m/2} \prod_{i=2,4,...m} \left( t_i+ t_{i-1}-|t_{i}-t_{i-1}|\right)
\]
using (\ref{eq:lower est}). The same estimate holds for $f_2(a,b;\mathbf{t})$. 
Hence by (\ref{eq:C-S}),
\begin{equation}\label{eq:upper est}
f(a,b;\mathbf{t}) \le f(0,b;\mathbf{t})+o\left( (b+1)^{-m/2} \right) \sim (b+1)^{-m/2} \prod_{i=2,4,...m} \left( t_i+ t_{i-1}-|t_{i}-t_{i-1}|\right).
\end{equation}
Combining (\ref{eq:lower est}) and (\ref{eq:upper est}) concludes the proof.
\end{proof}

\begin{Lem}\label{Lem:corner 1 limit cumulant}
Let $X_\rho(t)$ be the limit process in (\ref{eq:X_rho}).  For $m\ge 3$, 
\[
\kappa_m \left( \sum_{i=1}^n c_i X_\rho(t_i) \right)=
\begin{cases}
\rho^{m/2}  (m-1)! \left[\sum_{i,j=1}^n  c_ic_j~ \frac{1}{2} \left( |t_i| +|t_j| -|t_i-t_j| \right)\right]^{m/2} &\text{ if } m \text{ is }even;\\
 0 &\text{ if } m \text{ is }odd.
\end{cases}
\] 
\end{Lem}
\begin{proof}
Then because $B_1(t)$, $B_2(t)$ and $W$ are independent,
\[
\kappa_m \left( \sum_{i=1}^n c_i X_\rho(t_i) \right) 
= \kappa_m\left(\rho^{1/2} \sum_{i=1}^n c_i  W B(t_i)\right) + \kappa_m\left((1-\rho)^{1/2} \sum_{i=1}^n c_i B'(t_i)\right).
\]
Now note that the second term is Gaussian and thus the cumulants of order higher  than 2 is always zero. Applying Lemma \ref{Lem:Z_gamma limit cumulant} (with $\gamma=-1$)  to the first term concludes the proof.
\end{proof}

Now we proceed to the proof of Theorem \ref{Thm:main corner (-1/2,-1)}.
Again by Lemma \ref{Lem:tight}, tightness always holds. We only need to show the convergence of the finite-dimensional distributions.

\begin{proof}[Proof of $\ConvFDD$ in Theorem \ref{Thm:main corner (-1/2,-1)}.]
The distribution of $\sum_{i=1}^n c_i X_\rho(t_i)$ is moment-determinate since it is a second-order polynomial in normal random variables (see, e.g., \citet{slud:1993:moment}). One can therefore use a method of moments.

We analyze the asymptotics of the cumulants in (\ref{eq:joint cum gen Ronsen}) with $m\ge 3$ and $A(\gamma_1,\gamma_2)$ as given in (\ref{eq:A}) as $(\gamma_1,\gamma_2)\rightarrow (-1/2,-1)$. First, by Lemma \ref{Lem:beta asymp},
\begin{align}\label{eq:corner 1 A}
A(\gamma_1,\gamma_2)^m&\sim (\gamma_1+\gamma_2+3/2)^{m/2} \big[\B(1/2,1/2)\B(\gamma_2+1,1/2)+\B(1/2,-2\gamma_1-1)\B(\gamma_2+1,1)\big]^{-m/2}\notag\\
&\sim (\gamma_1+\gamma_2+3/2)^{m/2} \big[\B(1/2,-2\gamma_1-1)\B(\gamma_2+1,1)\big]^{-m/2}\notag\\
&\sim (\gamma_1+\gamma_2+3/2)^{m/2} (-2\gamma_1-1)^{m/2} (\gamma_2+1)^{m/2},
\end{align}
which converges to $0$.

Now we analyze the asymptotics of the terms  of $C_m(\gamma_1,\gamma_2;\mathbf{t},\mathbf{c})$ in (\ref{eq:C(t1,...,tk)}) as $\sigma$ varies in $\{1,2\}^m$.
 When $m$ is even, consider first the two main terms where 
\[\text{$\sigma=(1,2,1,2,\ldots,1,2)$ and $\sigma=(2,1,2,1,\ldots,2,1)$},\]
which correspond to $\#\{j:~ \sigma_j=\sigma_{j-1}'=1\}=m/2$. As in the proof of Theorem \ref{Thm:main edge}, the corresponding term when $\sigma=(1,2,1,2,\ldots,1,2)$
in (\ref{eq:C(t1,...,tk)}) (it is the same for $\sigma=(2,1,2,1,\ldots,2,1)$)  is
\begin{align}
 &\sum_{i_1,\ldots,i_m=1}^n c_{i_1}\ldots c_{i_m} \B(\gamma_1+1,-2\gamma_1-1)^{m/2}\B(\gamma_2+1,-2\gamma_2-1)^{m/2} \notag\\
&\times\int_0^{t_{i_1}}ds_1 \ldots \int_0^{t_{i_m}}ds_m 
 |s_1-s_m|^{2\gamma_1+1} |s_2-s_1|^{2\gamma_2+1} \ldots |s_{m-1}-s_
{m-2}|^{2\gamma_1+1} |s_m-s_{m-1}|^{2\gamma_2+1}\notag\\
&\sim  (-2\gamma_1-1)^{-m/2}(\gamma_2+1)^{-m} 
\left[\sum_{i,j=1}^n  c_ic_j~ \frac{1}{2} \left( |t_i| +|t_j| -|t_i-t_j| \right)\right]^{m/2},\label{eq:corner 1 main asymp}
\end{align}
where the last line is due to  Lemma \ref{Lem:beta asymp} and Lemma \ref{Lem:delicate asymp}. 
 
Any other $\sigma$ term in (\ref{eq:C(t1,...,tk)}) is negligible because it is of order $O\big((-2\gamma_1-1)^{-r}(\gamma_2+1)^{-m}\big)$, where 
\begin{equation}\label{eq:r count}
r=\#\{j:~ \sigma_j=\sigma_{j-1}'=1\}=\#\{j:~ \sigma_j=\sigma_{j-1}'=2\}<m/2.
\end{equation}
Indeed, let us suppose (\ref{eq:r count}) and examine a corresponding $\sigma$ term in the expansion of the product $\prod_{j=1}^m$ in  (\ref{eq:C(t1,...,tk)}). Call this term $P_m$.
In $P_m$, there are $r$ factors of
\begin{equation}\label{eq:term type 1}
\B(\gamma_1+1,-2\gamma_1-1)|s_j-s_{j-1}|^{2\gamma_1+1},
\end{equation}
and there are $r$ factors of
\begin{equation}\label{eq:term type 2}
\B(\gamma_2+1,-2\gamma_2-1)|s_j-s_{j-1}|^{2\gamma_2+1}.
\end{equation}
Since (\ref{eq:r count}) implies that
$\#\{j:~ \sigma_j\neq \sigma_{j-1}'\}=m-2r$,
 there are also $m-2r$ factors in $P_m$, which are  either 
\[
(s_j-s_{j-1})_+^{\gamma_1+\gamma_2+1}\B(\gamma_1+1,-\gamma_1-\gamma_2-1)+ (s_{j-1}-s_j)_+^{\gamma_1+\gamma_2+1}\B(\gamma_2+1,-\gamma_1-\gamma_2-1),
\]
or 
\[
(s_j-s_{j-1})_+^{\gamma_1+\gamma_2+1}\B(\gamma_2+1,-\gamma_1-\gamma_2-1)+ (s_{j-1}-s_j)_+^{\gamma_1+\gamma_2+1}\B(\gamma_1+1,-\gamma_1-\gamma_2-1).
\] 
These last two expressions are both bounded by
\begin{equation}\label{eq:term type 3}
|s_j-s_{j-1}|^{\gamma_1+\gamma_2+1}\big[\B(\gamma_2+1,-\gamma_1-\gamma_2-1)+\B(\gamma_1+1,-\gamma_1-\gamma_2-1)\big].
\end{equation}
In view of Lemma \ref{Lem:beta asymp}, the beta functions in
 (\ref{eq:term type 1}), (\ref{eq:term type 2}) and (\ref{eq:term type 3}) behave like $(-2\gamma_1-1)^{-1}$, $(\gamma_2+1)^{-1}$ and $(\gamma_2+1)^{-1}$ respectively. Therefore, the beta functions contribute an order
\[
(-2\gamma_1-1)^{-r}(\gamma_2+1)^{-r}(\gamma_2+1)^{-(m-2r)}=(-2\gamma_1-1)^{-r}(\gamma_2+1)^{-(m-r)}.\] 
The integrand involving $|s_{j-1}-s_j|^{2\gamma_2+1}$  contribute an order $(\gamma_2+1)^{-r}$. So the total order is $(-2\gamma_1-1)^{-r}(\gamma_2+1)^{-m}$.
These arguments can be rigorously justified by first applying the Cauchy-Schwartz  as in (\ref{eq:C-S}) to break the cyclic integrand, and  then bound   as in  (\ref{eq:delicate bound}).
Therefore in view of (\ref{eq:corner 1 main asymp}),  and after also including the case $\sigma=(2,1,2,1,\ldots,2,1)$, we conclude that
\begin{equation}\label{eq:corner 1 C}
C_m(\gamma_1,\gamma_2;\mathbf{t},\mathbf{c}) \sim 2 (-2\gamma_1-1)^{-m/2}(\gamma_2+1)^{-m} 
\left[\sum_{i,j=1}^n  c_ic_j~ \frac{1}{2} \left( |t_i| +|t_j| -|t_i-t_j| \right)\right]^{m/2},
\end{equation}
if $m$ is even.

When $m$ is odd, there are at most $(m-1)/2$ times of $\sigma_{j}=\sigma'_{j-1}=1$ or $\sigma_{j}=\sigma'_{j-1}=2$. It can be shown similarly that $C_m(\gamma_1,\gamma_2;\mathbf{t},\mathbf{c})$ is of the order 
\begin{equation}\label{eq:odd order}
(-2\gamma_1-1)^{-(m-1)/2}(\gamma_2+1)^{-m},
\end{equation}
 which  is dominated by the order of convergence to $0$ of $A(\gamma_1,\gamma_2)^m$ in (\ref{eq:corner 1 A}).
Now combining this fact with (\ref{eq:rho limit corner 1}), (\ref{eq:joint cum gen Ronsen}),  (\ref{eq:corner 1 A}) and (\ref{eq:corner 1 C}), we have  when $m$ is even,
\begin{align}
\kappa_m\left(\sum_{i=1}^n c_i Z_{\gamma_1,\gamma_2}(t_i)\right)&\sim
\left(\frac{\gamma_1+\gamma_2+3/2}{\gamma_2+1}\right)^{m/2} (m-1)! \left[\sum_{i,j=1}^n  c_ic_j~ \frac{1}{2} \left( |t_i| +|t_j| -|t_i-t_j| \right)\right]^{m/2}\label{eq:kappa_m asymp even corner}\\
&\rightarrow \rho^{m/2}  (m-1)! \left[\sum_{i,j=1}^n  c_ic_j~ \frac{1}{2} \left( |t_i| +|t_j| -|t_i-t_j| \right)\right]^{m/2},\notag
\end{align}
and when $m$ is odd,
\[
\kappa_m\left(\sum_{i=1}^n c_i Z_{\gamma_1,\gamma_2}(t_i)\right) \rightarrow 0.
\]
Now use  Lemma \ref{Lem:corner 1 limit cumulant} to identify the limit process.
\end{proof}

\subsection{Proof of Theorem \ref{Thm:main corner (-1/2,-1/2)}}

We state first a combinatorial result.
\begin{Lem}\label{Lem:count}
Let $\sigma=(\sigma_1,\ldots,\sigma_m)\in \{1,2\}^m$. Let $\sigma'=(\sigma_1',\ldots,\sigma_m')$ be the complement of $\sigma$, namely, $\sigma'_i=1$ if $\sigma_i=2$ and $\sigma'_i=2$ if $\sigma_i=1$, $i=1,\ldots,m$. 
 Let $\sigma_{0}$  be understood as $\sigma_m$ and let $\sigma_{0}'$  be understood as $\sigma_m'$. Then for a fixed integer $0\le r\le m/2$,
\begin{equation}\label{eq:sigma count}
 \#\Big\{\sigma\in \{1,2\}^m:~ \#\{j: ~\sigma_j=\sigma_{j-1}'=1\}=r\Big\}=2{m\choose 2r}.
\end{equation}
\end{Lem}
\begin{proof}
 If  $\sigma_{j-1}\neq \sigma_j$,  we say that there is an alternation at $j$. There are ${m\choose k}$ ways to place $k$ alternations.  The positions of the alternations determine the whole $\sigma$ up to the replacement of  $1$'s  into $2$'s and vice-versa. Hence there are $2{m\choose k}$ possible $\sigma$'s.
To relate $k$ to $r$, note that the relation $\sigma_j=\sigma_{j-1}'$ holds if and only if  $\sigma_{j-1}\neq \sigma_j$. Since 
\[r=\#\{j: ~\sigma_j=\sigma_{j-1}'=1\} =\#\{j:~\sigma_j=\sigma_{j-1}'=2\},\]
 we have 
\[
k=\#\{j:~\sigma_j\neq \sigma_{j-1}\}=\#\{j: ~\sigma_j=\sigma_{j-1}'=1\} +\#\{j:~\sigma_j=\sigma_{j-1}'=2\}=2r.
\]

\end{proof}

\begin{Lem}\label{Lem:corner 2 limit cumulant}
Let $Y_\rho(t)$ be the limit process in (\ref{eq:Y_rho}).  For $m\ge 3$, 
\begin{align}\label{eq:Y_rho cumulant}
\kappa_m \left( \sum_{i=1}^n c_i Y_\rho(t_i) \right)=&\frac{\left[(\rho+1)^{-1}+(2\sqrt{\rho})^{-1}\right]^m+\left[(\rho+1)^{-1}-(2\sqrt{\rho})^{-1}\right]^m}{\left[(\rho+1)^{-2}+(4\rho)^{-1}\right]^{m/2}} \notag\\
&\times\left(\sum_{i=1}^n c_it_i\right)^m \frac{(m-1)!}{2}.
\end{align}
\end{Lem}
\begin{proof}
Let
\begin{equation*}
a_\rho=  \frac{(\rho+1)^{-1}+(2\sqrt{\rho})^{-1}}{\sqrt{2(\rho+1)^{-2}+(2\rho)^{-1}}}, \quad b_\rho= \frac{(\rho+1)^{-1}-(2\sqrt{\rho})^{-1}}{\sqrt{2(\rho+1)^{-2}+(2\rho)^{-1}}}
\end{equation*}
Because $X_1$ and $X_2$ are two independent standardized $\chi_1^2$ random variables, we have
\begin{align*}
\kappa_m \left( \sum_{i=1}^n c_i Y_\rho(t_i) \right)
&=\kappa_m \left( \sum_{i=1}^n c_i t_i(a_\rho X_1+b\rho X_2) \right)
=\left(\sum_{i=1}^n c_i t_i\right)^m [\kappa_m(a_\rho X_1)+\kappa_m(b_\rho X_2)]\\
&= \left(\sum_{i=1}^n c_i t_i\right)^m(a_\rho^m+b_\rho^m) \kappa(X_1)= 2^{m/2}(a_\rho^m+b_\rho^m)\left(\sum_{i=1}^n c_i t_i\right)^m  \frac{(m-1)!}{2}. 
\end{align*}
The factor $2^{m/2}(a_\rho^m+b_\rho^m)$ can be rewritten as the first factor in (\ref{eq:Y_rho cumulant}).
\end{proof}

Note that $a+b\sim A+B$ for $a,b,A,B>0$, if $a\sim A$, $b\sim B$ and $a/b\sim \lambda$, where  $\lambda$ is a fixed number from $0$ to $\infty$ (can be $\infty$), as will always be the case under our assumptions.

We now prove Theorem \ref{Thm:main corner (-1/2,-1/2)}.  In view of Lemma \ref{Lem:tight}, we only need to show the convergence of the finite-dimensional distributions.
\begin{proof}[Proof of $\ConvFDD$ in Theorem \ref{Thm:main corner (-1/2,-1/2)}]
We can  use a method of moments again because  the limit $\sum_{i=1}^n c_i Y_\rho(t_i)$ is a second-order polynomial in normal random variables. We analyze the asymptotics of the cumulants in (\ref{eq:joint cum gen Ronsen}) with $m\ge 3$ and $A(\gamma_1,\gamma_2)$  in (\ref{eq:A}) as $(\gamma_1,\gamma_2)\rightarrow (-1/2,-1/2)$.
Lemma \ref{Lem:beta asymp} yields
\begin{equation}\label{eq:corner 2 A}
A(\gamma_1,\gamma_2)^m \sim \left[(-\gamma_1-\gamma_2-1)^{-2}+ (-2\gamma_1-1)^{-1}(-2\gamma_2-1)^{-1} \right]^{-m/2},
\end{equation}
and $C_m$ in (\ref{eq:C(t1,...,tk)}) satisfies
\begin{equation}\label{eq:corner 2 C}
C_m(\gamma_1,\gamma_2;\mathbf{t},\mathbf{c})\sim \left(\sum_{i=1}^n c_i t_i\right)^m  \sum_{\sigma\in\{1,2\}^m} \prod_{j=1}^m (-\gamma_{\sigma_j}-\gamma_{\sigma_{j-1}'}-1)^{-1},
\end{equation}
where we get the term $\left(\sum_{i=1}^n c_i t_i\right)^m$ from 
$\sum_{i_1,\ldots,i_m=1}^n  c_{i_1}\ldots c_{i_m} \int_{0}^{t_{i_1}}ds_1 \ldots \int_{0}^{t_{i_m}}ds_m.
$

Let $r=\#\{j: ~\sigma_j=\sigma_{j-1}'=1\}=\#\{j: ~\sigma_j=\sigma_{j-1}'=2\}$. Then using Lemma \ref{Lem:count}, we can write
\begin{align}\label{eq:corner 2 sum}
\sum_{\sigma\in\{1,2\}^m}\prod_{j=1}^m (-\gamma_{\sigma_j}- \gamma_{\sigma_{j-1}'}-1)^{-1}
=\sum_{0\le r\le m/2} 2{m \choose 2r} (-2\gamma_1-1)^{-r}(-2\gamma_2-1)^{-r} (-\gamma_1-\gamma_2-1)^{-(m-2r)}. 
\end{align}
Hence by (\ref{eq:joint cum gen Ronsen}), (\ref{eq:corner 2 A}), (\ref{eq:corner 2 C}) and (\ref{eq:corner 2 sum}), one has
\begin{equation}\label{eq:corner 2 cum inter}
\kappa_m\left(\sum_{i=1}^n c_i Z_{\gamma_1,\gamma_2}(t_i)\right)\sim  (m-1)! \left(\sum_{i=1}^n c_i t_i\right)^m \sum_{0\le r\le m/2} {m\choose 2r}  U(\gamma_1,\gamma_2;m,r).
\end{equation}
where
\[
U(\gamma_1,\gamma_2;m,r):=\frac{(-2\gamma_1-1)^{-r}(-2\gamma_2-1)^{-r} (-\gamma_1-\gamma_2-1)^{-(m-2r)}}{\left[(-\gamma_1-\gamma_2-1)^{-2}+ (-2\gamma_1-1)^{-1}(-2\gamma_2-1)^{-1}\right]^{m/2}}.
\]
As $(\gamma_1,\gamma_2)\rightarrow (-1/2,-1/2)$ and $(\gamma_1+1/2)/(\gamma_2+1/2)\rightarrow \rho \in [0,1]$, in the case $\rho>0$, some elementary calculation shows
\begin{equation}\label{eq:corner 2 U asymp}
U(\gamma_1,\gamma_2;m,r)\rightarrow \frac{\left[1/(2\sqrt{\rho})\right]^{2r} \left[1/(\rho+1)\right]^{m-2r}}{\left[(\rho+1)^{-2}+(4\rho)^{-1}\right]^{m/2}}, 
\end{equation}
and in the case $\rho=0$,  
\begin{equation}\label{eq:U 0-1}
U(\gamma_1,\gamma_2;m,r)\rightarrow \begin{cases}
1 &\text{ if }r=m/2 \text{ ($m$ must be even in this case)};\\
0 &\text{ if } r<m/2.
\end{cases}
\end{equation}
This expression (\ref{eq:U 0-1}) also coincides with the limit in (\ref{eq:corner 2 U asymp}) as $\rho\rightarrow 0$.  In the argument below we omit the case $\rho=0$, which can be either treated separately, or obtained by taking the limit  as $\rho\rightarrow 0$.  

Set $a=1/(2\sqrt{\rho})$ and $b=1/(\rho+1)$.
Using the identity $(a+b)^m+(a-b)^m = \sum _{0\le r\le m/2}2 {m \choose 2r} a^{2r} b^{m-2r}$, one can write following (\ref{eq:corner 2 cum inter}) and (\ref{eq:corner 2 U asymp}) that
\[
\kappa_m\left(\sum_{i=1}^n c_i Z_{\gamma_1,\gamma_2}(t_i)\right)\rightarrow  \frac{(a+b)^m - (a-b)^m}{(a^2+b^2)^{m/2}} \left(\sum_{i=1}^n c_i t_i\right)^m \frac{(m-1)!}{2},
\]
which is (\ref{eq:Y_rho cumulant}).
Now use Lemma \ref{Lem:corner 2 limit cumulant} to identify the limit process, concluding the proof.
\end{proof}

\noindent\textbf{Additional results}
\smallskip
 
We deal now with the following additional three points:
\begin{enumerate}
\item We show that the weak convergence  proved in the previous theorems \emph{cannot} be strengthened to  convergence in $L^2(\Omega)$ nor even in probability; 
\medskip

\item We apply the results of  \citet{nourdin:peccati:2013:optimal} and \citet{eichelsbacher:thale:2014:malliavin} to determine the rate of convergence on the boundaries  $d$ and $e_1$ (or $e_2$);
\medskip

\item We include an alternate proof of Theorem \ref{Thm:main edge} in the spirit of Remark \ref{Rem:physical}  which provides further insight on the convergence.
\end{enumerate}

\section{No  convergence in $L^2(\Omega)$}\label{sec:add 1}
The generalized Rosenblatt process $Z_{\gamma_1,\gamma_2}(t)$ was defined in (\ref{eq:ori gen Ros proc}) (see also (\ref{eq:gen Rosenblatt proc})).  
We have shown weak convergence (convergence in distribution) for the generalized Rosenblatt process $Z_{\gamma_1,\gamma_2}(t)$ in previous theorems. Is it possible that some of these convergences are actually in a stronger mode, say, in probability? We provide a negative answer here.
\begin{Thm}\label{Thm:no L^2}
In Theorem \ref{Thm:main diag}, \ref{Thm:main edge}, \ref{Thm:main corner (-1/2,-1)} and \ref{Thm:main corner (-1/2,-1/2)}, the weak convergence cannot be extended to  convergence in $L^2(\Omega)$, nor even to convergence  in probability.
\end{Thm}  
\begin{Rem}
In fact, it suffices to show that the convergence cannot be extended to convergence in $L^2(\Omega)$. This is because,
on a fixed order Wiener chaos, convergence in $L^2(\Omega)$ and convergence in probability are equivalent. See \citet{schreiber:1969:fermeture}. Alternatively, to verify the equivalence, suppose that $X_n$ is a sequence on a fixed order Wiener chaos, and $X_n$ converges in probability to $X$. The sequence is therefore tight. Then by, e.g, Lemma 2.1(ii) of \citet{nourdin:rosinski:2012:asymptotic}, $\sup_n\E |X_n|^{p}<\infty$ for any $p>0$, which entails uniform integrability and hence convergence in $L^2(\Omega)$.
\end{Rem}
To prove Theorem \ref{Thm:no L^2}, it suffices to show that any sequence of
\[
Z_{\gamma_1,\gamma_2}:=Z_{\gamma_1,\gamma_2}(1)
\]
as $(\gamma_1,\gamma_2)$ approach the boundaries is \emph{not} a Cauchy sequence in $L^2(\Omega)$. 
Let $(\alpha_1,\alpha_2)$ and $(\gamma_1,\gamma_2)$ be in the region  $\mathbf{\Delta}$ in (\ref{eq:region}). Then since $Z_{\gamma_1,\gamma_2}$ is standardized, we have
\begin{equation}\label{eq:L^2 square distance}
\E \left(Z_{\alpha_1,\alpha_2}-Z_{\gamma_1,\gamma_2}\right)^2
=2-2\E Z_{\alpha_1,\alpha_2} Z_{\gamma_1,\gamma_2}.
\end{equation}
If $(\alpha_1,\alpha_2)$ and $(\gamma_1,\gamma_2)$ converge to the same point on the boundary, we may  expect that $\E Z_{\alpha_1,\alpha_2} Z_{\gamma_1,\gamma_2} \rightarrow 1$ and hence $\E \left(Z_{\alpha_1,\alpha_2}-Z_{\gamma_1,\gamma_2}\right)^2\rightarrow 0$, which would prove Cauchy convergence. We will show, however, that 
\begin{equation}\label{eq:liminf < 1}
\liminf_{(\alpha_1,\alpha_2),(\gamma_1,\gamma_2)\rightarrow \text{ boundary point }} \E Z_{\alpha_1,\alpha_2} Z_{\gamma_1,\gamma_2}<1.
\end{equation}
In other words, we will show that there is no $L^2(\Omega)$ continuity at the boundary.

First we compute the covariance in (\ref{eq:L^2 square distance}).
\begin{Lem}
\begin{align}\label{eq:cov bw para}
\E Z_{\alpha_1,\alpha_2} Z_{\gamma_1,\gamma_2}
=&A(\alpha_1,\alpha_2)A(\gamma_1,\gamma_2) (\alpha_1+\alpha_2+\gamma_1+\gamma_2+3)^{-1}(\alpha_1+\alpha_2+\gamma_1+\gamma_2+4)^{-1}\notag\\
 \times\big[&\B(\alpha_1+1,-\alpha_1-\gamma_1-1) \B(\alpha_2+1,-\alpha_2-\gamma_2-1)\notag\\+&\B(\gamma_1+1,-\alpha_1-\gamma_1-1)\B(\gamma_2+1,-\alpha_2-\gamma_2-1)\notag\\
   +&\B(\alpha_2+1,-\alpha_2-\gamma_1-1) \B(\alpha_1+1,-\alpha_1-\gamma_2-1)\notag\\+&\B(\gamma_1+1,-\alpha_2-\gamma_1-1)\B(\gamma_2+1,-\alpha_1-\gamma_2-1)\big].
\end{align}
\begin{proof}
We shall use the representation (\ref{eq:gen Rosenblatt proc}) of $Z_{\gamma_1,\gamma_2}(t)$ in order to apply the formula 
\[
\E I_2(f) I_2(g)=2\langle f,g \rangle_{L^2(\mathbb{R}^2)}
\]
for symmetric functions $f$ and $g$ (see (7.3.39) of \citet{peccati:taqqu:2011:wiener}). Using (\ref{eq:int formula}), we get
\begin{align*}
 2 A (\alpha_1, \alpha_2&)^{-1}A(\gamma_1,\gamma_2)^{-1}  \E Z_{\alpha_1,\alpha_2} Z_{\gamma_1,\gamma_2}\\
=&\int_{[0,1]^2}d\mathbf{s} \int_{\mathbb{R}^2}d\mathbf{x} \left[(s_1-x_1)_+^{\alpha_1}(s_1-x_2)_+^{\alpha_2}+(s_1-x_1)_+^{\alpha_2}(s_1-x_2)_+^{\alpha_1}\right] \\
&\qquad \qquad\quad~~  \times \left[(s_2-x_1)_+^{\gamma_1}(s_2-x_2)_+^{\gamma_2}+(s_2-x_1)_+^{\gamma_2}(s_2-x_2)_+^{\gamma_1}\right] \\
=& 2 \int_{[0,1]^2}d\mathbf{s}\Big[ (s_2-s_1)_+^{\alpha_1+\alpha_2+\gamma_1+\gamma_2+2}\B(\alpha_1+1,-\alpha_1-\gamma_1-1) \B(\alpha_2+1,\alpha_2-\gamma_2-1)\\
&\qquad ~\quad  +(s_1-s_2)_+^{\alpha_1+\alpha_2+\gamma_1+\gamma_2+2}\B(\gamma_1+1,-\alpha_1-\gamma_1-1)\B(\gamma_2+1,-\alpha_2-\gamma_2-1)\\
&\qquad ~\quad + (s_2-s_1)_+^{\alpha_1+\alpha_2+\gamma_1+\gamma_2+2}\B(\alpha_2+1,-\alpha_2-\gamma_1-1)\B(\alpha_1+1,-\alpha_1-\gamma_2-1)\\
&\qquad ~\quad +(s_1-s_2)_+^{\alpha_1+\alpha_2+\gamma_1+\gamma_2+2}\B(\gamma_1+1,-\alpha_2-\gamma_1-1)\B(\gamma_2+1,-\alpha_2-\gamma_2-1)\Big]
\end{align*}
Since  $\alpha_1+\alpha_2>-3/2$ and $\gamma_1+\gamma_2>-3/2$, we have $\alpha_1+\alpha_2+\gamma_1+\gamma_2+2>-1$. Since 
\[ 
\int_{[0,1]^2}(s_1-s_2)_+^{u}d\mathbf{s}=\int_{[0,1]^2} (s_2-s_1)_+^{u}d\mathbf{s}= (u+1)^{-1}(u+2)^{-1} \] 
for $u>-1$, we  get (\ref{eq:cov bw para}).
\end{proof}

\bigskip
\begin{proof}[Proof of Theorem \ref{Thm:no L^2}]
\quad \\

\emph{Case of Theorem \ref{Thm:main diag}.}
By (\ref{eq:boundary two}), an element of the second chaos converges in distribution to a Gaussian. That this cannot be extended to convergence in $L^2(\Omega)$ follows from the fact  that
 $\{I_2(f): f\in L^2(\mathbb{R}^2)\}$ is a closed subspace in $L^2(\Omega)$.  Hence the $L^2(\Omega)$ limit of a double Wiener-It\^o integral must still be a double Wiener-It\^o integral, which means that it cannot be Gaussian. 

\medskip
\emph{Case of Theorem \ref{Thm:main edge}.}
Let $(\alpha_1,\alpha_2)\rightarrow (-1/2,\gamma)$ and $(\gamma_1,\gamma_2)\rightarrow (-1/2,\gamma)$, where $\gamma\in (-1,-1/2).$  Assume in addition that the convergence speeds are comparable, that is,  $(\alpha_1+1/2)/(\gamma_1+1/2)\sim r \in (0,1)$. Then using (\ref{eq:A(gamma)^m}) with $m=1$, Lemma \ref{Lem:beta asymp}, and (\ref{eq:cov bw para}),  one has
\begin{align*}
\E Z_{\alpha_1,\alpha_2} Z_{\gamma_1,\gamma_2}\sim&  (-2\alpha_1-1)^{1/2} (-2\gamma_1-1)^{1/2} (2\gamma+3)(\gamma+1)  \B(\gamma+1,-2\gamma-1)^{-1} \\
&\times (2+2\gamma)^{-1}(3+2\gamma)^{-1} \big[2\B(\gamma+1,-2\gamma-1)(-\alpha_1-\gamma_1-1)^{-1}\big]\\
\sim& \frac{(-2\alpha_1-1)^{1/2}(-2\gamma_1-1)^{1/2}}{(-\alpha_1-\gamma_1-1)}\sim 2r^{1/2}/(1+r)<1.
\end{align*}
\medskip

\emph{Case of Theorem \ref{Thm:main corner (-1/2,-1)}.}
When $\rho<1$, the limit in (\ref{eq:X_rho}) involves a Gaussian component, which by the same reason as in ``Case of Theorem \ref{Thm:main diag}'' implies that $L^2(\Omega)$ convergence cannot hold. We only need to consider the case $\rho=1$.

We therefore suppose that  $(\alpha_1,\alpha_2)\rightarrow (-1/2,-1)$ and $(\gamma_1,\gamma_2)\rightarrow (-1/2,-1)$ and that $\rho=1$, that is by (\ref{eq:rho limit corner 1}), that $(\alpha_1+1/2)/(\alpha_2+1)\rightarrow 0$ and $ (\gamma_1+1/2)/(\gamma_2+1)\rightarrow 0$. Assume in addition that $(\alpha_1+1/2)/(\gamma_1+1/2)\sim (\alpha_2+1)/(\gamma_2+1)\sim r \in (0,1)$. By (\ref{eq:corner 1 A}) with $m=1$,  Lemma \ref{Lem:beta asymp}, and (\ref{eq:cov bw para}), we have
\begin{align*}
\E Z_{\alpha_1,\alpha_2} Z_{\gamma_1,\gamma_2}\sim&(\alpha_1+\alpha_2+3/2)^{1/2} (-2\alpha_1-1)^{1/2} (\alpha_2+1)^{1/2}(\gamma_1+\gamma_2+3/2)^{1/2} (-2\gamma_1-1)^{1/2} (\gamma_2+1)^{1/2} \\
&\times(\alpha_1+\alpha_2+\gamma_1+\gamma_2+3)^{-1}(-\alpha_1-\gamma_1-1)^{-1}[(\alpha_2+1)^{-1}+(\gamma_2+1)^{-1}]\\
\sim& \frac{(\alpha_2+1)(-2\alpha_1-1)^{1/2}(\gamma_2+1)(-2\gamma_1-1)^{1/2}}{(\alpha_2+1+\gamma_2+1)(-\alpha_1-\gamma_1-1)}[(\alpha_2+1)^{-1}+(\gamma_2+1)^{-1}] 
\\ 
\sim& 2r^{1/2}/(r+1)<1.
\end{align*}

\emph{Case of Theorem \ref{Thm:main corner (-1/2,-1/2)}.}
Suppose that  $(\alpha_1,\alpha_2)\rightarrow (-1/2,-1/2)$ and $(\gamma_1,\gamma_2)\rightarrow (-1/2,-1/2)$ and that $(\alpha_1+1/2)/(\alpha_2+1/2)\sim  (\gamma_1+1/2)/(\gamma_2+1/2) \sim \rho$,  where $\rho\in [0,1]$. Assume in addition that $(\alpha_1+1/2)/(\gamma_1+1/2)\sim (\alpha_2+1/2)/(\gamma_2+1/2)\sim  r\in (0,1)$. We apply (\ref{eq:corner 2 A}) with $m=1$, (\ref{eq:cov bw para}) and  Lemma \ref{Lem:beta asymp}. In this case,  all beta functions in (\ref{eq:cov bw para}) blow up and we get
\begin{align*}
\E Z_{\alpha_1,\alpha_2} Z_{\gamma_1,\gamma_2}\sim&  \left[(-\alpha_1-\alpha_2-1)^{-2}+ (-2\alpha_1-1)^{-1}(-2\alpha_2-1)^{-1} \right]^{-1/2}  \\
& \times \left[(-\gamma_1-\gamma_2-1)^{-2}+ (-2\gamma_1-1)^{-1}(-2\gamma_2-1)^{-1} \right]^{-1/2} \times \frac{1}{2}
 \\&  \times \big[2(-\alpha_1-\gamma_1-1)^{-1} (-\alpha_2-\gamma_2-1)^{-1} + 2(-\alpha_2-\gamma_1-1)^{-1}(-\alpha_1-\gamma_2-1)^{-1}\big]
\\\sim& \frac{4r}{(r+1)^2}\left(\frac{(r+\rho)(1+r\rho)+(r+1)^2\rho}{(1+\rho)^2+4\rho }\right) \frac{(1+\rho)^2}{(r+\rho)(1+r\rho)},
\end{align*}
which is   close to zero if $r$ is small. Thus (\ref{eq:liminf < 1}) holds.
\end{proof}
\end{Lem}

%\begin{Rem}
%The results of this subsection conveys the following message: by relying on convergence of cumulants alone, one cannot distinguish between weak convergence and $L^2$ convergence for random variables on the second Wiener chaos. The form of the kernel function matters. 
%
%For example, if one takes $g_n(x)=I\{(-1,0)\times (0,1)\}$. Then 
%\[
%G_n:=\int_{\mathbb{R}^2}' g_n(x_1,x_2)B(dx_1)B(dx_2)=[B(1)-B(0)]\times [B(0)-B(-1)]\overset{d}{=}W B,
%\]
% where $W$ and $B$ are two 
%independent standard normal random variables. By Theorem \ref{Thm:main edge}, we have an example of $f_n\in L^2(\mathbb{R})$, such that 
%\[F_n:=\int_{\mathbb{R}^2}' f_n(x_1,x_2)B(dx_1)B(dx_2) \ConvD W B
%\]
%as $n\rightarrow\infty$, but $F_n$ does \emph{not} converge in $L^2$. In both cases, the cumulants  converge to the cumulants of $WB$. But  $G_n$ converges in $L^2$ while $F_n$ converges only weakly. 
%
%Similar examples  can  be created from Theorem \ref{Thm:main corner (-1/2,-1/2)}, where the limit is a linear combination of independent centered chi squares.
%\end{Rem}

\section{Convergence rate of marginal distribution on the boundaries}\label{sec:add 2}
Rates of convergence  of the marginal distribution of multiple Wiener-It\^o integrals are available when the limit is Gaussian or is a product of independent Gaussians. We can thus apply these rates when converging to the boundaries of the triangle, with some corners excluded.

First we consider the  convergence rate of the marginal distribution in the case of Theorem \ref{Thm:main diag} and  \ref{Thm:main corner (-1/2,-1)} and the limit being Gaussian. 
We use  the notation $A\asymp B$, where $A$ and $B$ are two nonnegative quantities, to denote that there exist constants $c<C$ independent of $A$ and $B$ such that $cB\le A \le C B$.  Let $d_{TV}(X,Y)$ denote the total variation distance between the distributions of random variables $X$ and $Y$, namely
\[
d_{TV}(X,Y)=\sup_{S \in \mathcal{B}(\mathbb{R})}|P(X\in S)-P(Y\in S)|,
\]
where $\mathcal{B}(\mathbb{R})$ denotes the Borel sets on $\mathbb{R}$.

In \citet{nourdin:peccati:2013:optimal} Theorem 1.2, the following result was established:
\begin{Lem}\label{Lem:nourdin}
Let  $\{F_{\gamma}: \gamma\in G\subset \mathbb{R}^k\}$ be a family of random variables defined on a fixed-order Wiener chaos satisfying $\E F_\gamma^2=1$, where $G$ is an open set of indices. Suppose that the third cumulant  $\kappa_3(F_\gamma)$ and the fourth cumulant $\kappa_4(F_\gamma)$ converge uniformly to zero as $\gamma\in G$ approaches a set $E\subset \overline{G}$ (as the distance between the point $\gamma$ and the set $E$ converges to zero).  
Then there exits a neighborhood $\mathcal{N}(E)$ of  $E$ in $\mathbb{R}^k$, such that when $\gamma\in \mathcal{N}(E) \cap G $, we have  
\begin{equation}\label{eq:d tv}
d_{TV}(F_\gamma,N)\asymp M(F_\gamma) ,
\end{equation}
where $N$ is a standard normal random variable and
\begin{equation}\label{eq:max}
M(F_\gamma)= \max\left(|\E F_\gamma^3|,|\E F_\gamma^4-3| \right)=\max\left(|\kappa_3(F_\gamma)|,|\kappa_4(F_\gamma)| \right).
\end{equation}
\end{Lem}
\begin{Rem}
Though the theorem was originally stated in \citet{nourdin:peccati:2013:optimal} for a sequence $\{F_n\}$ with a discrete parameter $n$, examining the proof there one sees that for (\ref{eq:d tv}) to hold, one only needs  $\kappa_3(F_\gamma)$ and $\kappa_4(F_\gamma)$ to converge uniformly to zero, which is implied by our statement of the theorem.
\end{Rem}
\begin{Rem}\label{Rem:d_B}
Earlier in \cite{bierme:hermine:2012:optimal}, the same result (\ref{eq:d tv}) was established for the following distributional distance $d_B(\cdot,\cdot)$:
\begin{equation}\label{eq:d_B}
d_B(X,Y)=\sup_{h\in \mathcal{U}}\{|\E h(X)-\E h(Y)|\},
\end{equation}
where $\mathcal{U}$ is the class  of functions that are twice differentiable with continuous derivatives satisfying $\|h''\|_\infty<\infty$.
\end{Rem}

\begin{figure}[t,h]
\centering
\begin{tikzpicture}[scale=8]
\draw (-1/2,-1/2)-- (-1,-1/2);
\draw (-1,-1/2)-- (-1/2,-1);
\draw (-1/2,-1)-- (-1/2,-1/2);
\draw [ultra thick] (-1/2-0.07,-1+0.07)--(-1+0.07,-1/2-0.07);
\node at (-0.45, -0.45) {\small $(-\frac{1}{2},-\frac{1}{2})$};
\node at (-1/2,-1.05) {\small $(-\frac{1}{2},-1)$};
\node at (-1, -0.45) {\small $(-1,-\frac{1}{2})$};
\node at (-3/4-0.1,-3/4-0.11) {\small $\mathcal{N}(D_\epsilon)\cap \boldsymbol{\Delta}$};
\draw (-3/4-0.07,-3/4-0.07)--(-3/4+0.02,-3/4+0.02);
\node at (-3/4-0.25,-3/4+0.1) {\small $D_\epsilon$};
\draw (-3/4-0.22,-3/4+0.11)--(-3/4-0.1,-3/4+0.1);
\draw[dashed] (-0.875,-1/2-0.05)--(-1/2-0.05,-0.875);
\draw[dashed] (-1/2-0.05,-0.875)--(-1/2-0.05,-1+0.05);
\draw[dashed] (-0.875,-1/2-0.05)--(-1+0.05,-1/2-0.05);
\end{tikzpicture}
\caption{Illustration of the neighborhood $\mathcal{N}(D_\epsilon)$ of $D_\epsilon$ in Theorem \ref{Thm:diag}}\label{fig:pro diag}
\end{figure}
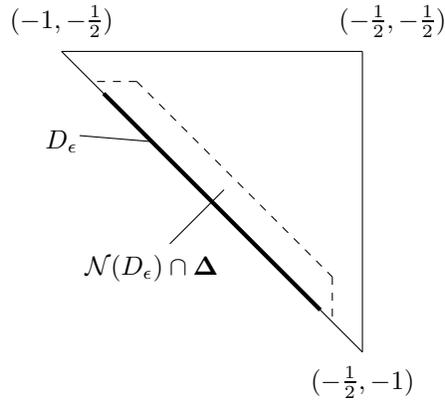

In the case of Theorem  \ref{Thm:main diag}, we considered convergence to the boundary $d$ through the neighborhood $\mathcal{N}(D_\epsilon)\cap \mathbf{\Delta}$   
illustrated in Figure \ref{fig:pro diag}.
Applying Lemma \ref{Lem:nourdin}, we get the following:
\begin{Thm}\label{Thm:diag}
Let $Z_{\gamma_1,\gamma_2}=Z_{\gamma_1,\gamma_2}(1)$, and let $N$ be a standard normal random variable. Then under the assumptions of Theorem \ref{Thm:main diag}, there exists a neighborhood  $\mathcal{N}(D_\epsilon)$ of the diagonal line segment $D_\epsilon:=\{\gamma_1+\gamma_2 +3/2=0:  \gamma_1,\gamma_2>-1+\epsilon\}$, such that when $(\gamma_1,\gamma_2)\in \mathcal{N}(D_\epsilon)\cap \boldsymbol{\Delta}$, we\footnote{Since $\boldsymbol{\Delta}$ is an open set,  $\mathcal{N}(D_\epsilon)\cap \boldsymbol{\Delta}$ does not contain the segment $D_\epsilon$.} have
\begin{equation}\label{eq:d tv goal}
d_{TV}( Z_{\gamma_1,\gamma_2} , N)\asymp (\gamma_1+\gamma_2+3/2)^{3/2}.
\end{equation}
\end{Thm}
\begin{proof}
Since $N$ is Gaussian, we can apply Lemma  \ref{Lem:nourdin}. To do so, we need to compute the cumulants $\kappa_3$ and $\kappa_4$ which are given in Proposition \ref{Pro:joint cum}.
We examine the relation (\ref{eq:joint cum gen Ronsen}) of Proposition \ref{Pro:joint cum} with $A=A(\gamma_1,\gamma_2)$ given in (\ref{eq:A}), $m=1$, $t=1$, and $c=1$.  The factor $C_m(\gamma_1,\gamma_2,1,1)$ in (\ref{eq:C(t1,...,tk)}) is a positive continuous function with respect to $(\gamma_1,\gamma_2)$. This can be shown by the Dominated Convergence Theorem as in Lemma \ref{Lem:bound circular}. Under the assumption of Theorem \ref{Thm:main diag}, the parameter $(\gamma_1,\gamma_2)$ is restricted away from boundary. So $C_m(\gamma_1,\gamma_2,1,1)$ is bounded below away from zero and bounded above away from infinity, and so are the factors in (\ref{eq:A}) except $[2(\gamma_1+\gamma_2)+3]^{1/2}$, which goes to zero as $\gamma_1+\gamma_2\rightarrow -3/2$.
 We get
\begin{equation}\label{eq:k_m rate}
\kappa_m(Z_{\gamma_1,\gamma_2})\asymp A(\gamma_1,\gamma_2)^m \asymp (\gamma_1+\gamma_2+3/2)^{m/2}, \qquad m\ge 3.
\end{equation}
The maximum in (\ref{eq:max}) is then  $\kappa_3(F_\gamma)$.
Combining this with (\ref{eq:d tv}), we get (\ref{eq:d tv goal}).
\end{proof}

From (\ref{eq:k_m rate}) and (\ref{eq:d tv}), it is the third cumulant that determines the rate of convergence in the case of Theorem \ref{Thm:main diag}.
When $(\gamma_1,\gamma_2)$ is allowed to be close to the corner $(-1/2,-1)$, that is, in the case of Theorem \ref{Thm:main corner (-1/2,-1)} when $\rho=0$, we will show that the fourth cumulant may come into play in the rate of convergence. 
\begin{Thm}\label{Thm:corner}
Let $Z_{\gamma_1,\gamma_2}=Z_{\gamma_1,\gamma_2}(1)$, and let $N$ be a standard normal random variable. Then under the assumptions of Theorem \ref{Thm:main corner (-1/2,-1)} when $\rho=0$, that is when
\begin{equation}\label{eq:corner asymp equiv}
-\gamma_1-1/2\sim \gamma_2+1, 
\end{equation}
 there exits a neighborhood $\mathcal{N}$ of $(-1/2,-1)$, such that when $(\gamma_1,\gamma_2)\in \mathcal{N}\cap \boldsymbol{\Delta}$, we have\footnote{As before, since $\boldsymbol{\Delta}$ is an open set,  $\mathcal{N}\cap \boldsymbol{\Delta}$ does not contain the limit point $(-1/2,-1)$.}
\begin{align}\label{eq:d tv goal corner}
d_{TV}( Z_{\gamma_1,\gamma_2} , N)&\asymp (\gamma_1+\gamma_2+3/2)^{3/2}(\gamma_2+1)^{-1}\big(1+ L(\gamma_1,\gamma_2) \big), 
\end{align}
 as $(\gamma_1,\gamma_2)\rightarrow (-1/2,-1)$, where
\begin{equation}\label{eq:L(gamma)}
L(\gamma_1,\gamma_2)=\sqrt{(-\gamma_1-1/2)^{-1}-(\gamma_2+1)^{-1}}=o\left((-\gamma_1-1/2)^{-1/2}\right)\text{ \emph{or} }~o\left((\gamma_2+1)^{-1/2} \right).
\end{equation}
\end{Thm}
\begin{proof}
First in view of (\ref{eq:rho limit corner 1}) with $\rho=0$, we have
\[
V(\gamma_1,\gamma_2):=(\gamma_1+\gamma_2+3/2)^{3/2}(\gamma_2+1)^{-1}\rightarrow 0,\quad \text{as }(\gamma_1,\gamma_2)\rightarrow (-1/2,-1).
\]
By  (\ref{eq:joint cum gen Ronsen}), (\ref{eq:corner 1 A}), (\ref{eq:odd order}) with $m=3$, and (\ref{eq:corner asymp equiv}), 
we get for the third cumulant
\begin{equation}\label{eq:kappa_3 rate}
\kappa_3(Z_{\gamma_1,\gamma_2}) \asymp (-\gamma_1-1/2)^{1/2}(\gamma_1+\gamma_2+3/2)^{3/2}(\gamma_2+1)^{-3/2}\sim V(\gamma_1,\gamma_2).
\end{equation}
By (\ref{eq:kappa_m asymp even corner}) with $m=4$ and also (\ref{eq:corner asymp equiv}), we have for the fourth cumulant
\begin{align}\label{eq:kappa_4 rate}
\kappa_4(Z_{\gamma_1,\gamma_2}) &\asymp \left(\frac{\gamma_1+\gamma_2+3/2} {\gamma_2+1}\right)^2
\sim V(\gamma_1,\gamma_2) \left(\frac{\gamma_1+\gamma_2+3/2}{(-\gamma_1-1/2)(\gamma_2+1)} \right)^{1/2}= V(\gamma_1,\gamma_2) L(\gamma_1,\gamma_2).
\end{align}
Since $\max(x,y)\asymp x+y$ for $x,y\ge 0$, we get
\begin{align*}
\max\left[\kappa_3(\gamma_1,\gamma_2),\kappa_4(\gamma_1,\gamma_2)\right]
  \asymp V(\gamma_1,\gamma_2) \left[ 1+ L(\gamma_1,\gamma_2)\right].
\end{align*}
We thus apply Lemma \ref{Lem:nourdin} to get (\ref{eq:d tv goal corner}).
At last, note that (\ref{eq:corner asymp equiv}) entails that
\[
L(\gamma_1,\gamma_2)=(-\gamma_1-1/2)^{-1/2}\sqrt{1-\frac{-\gamma_1-1/2}{\gamma_2+1}}=o\left((-\gamma_1-1/2)^{-1/2}\right)\text{ or }o\left((\gamma_2+1)^{-1/2}\right).
\]
\end{proof}

\begin{Rem}
In view of Remark \ref{Rem:d_B}, Theorem \ref{Thm:diag} and \ref{Thm:corner} also hold if the distance $d_{TV}(\cdot,\cdot)$ is replaced by the distance $d_B(\cdot,\cdot)$ defined by (\ref{eq:d_B}).
\end{Rem}

\begin{Rem}
The rate of convergence to zero in (\ref{eq:d tv goal corner}) is always slower than that of (\ref{eq:d tv goal}), which is expected since the corner $(-1/2,-1)$ also belongs to the non-Gaussian boundary. 
\end{Rem}
\begin{Rem}
From (\ref{eq:kappa_3 rate}) and (\ref{eq:kappa_4 rate}), one has
\[
\frac{\kappa_4(Z_{\gamma_1,\gamma_2})}{\kappa_3(Z_{\gamma_1,\gamma_2})}\asymp 
\sqrt{(-\gamma_1-1/2)^{-1}-(\gamma_2+1)^{-1}}=L(\gamma_1,\gamma_2),
\]
which is the term (\ref{eq:L(gamma)}) appearing in (\ref{eq:d tv goal corner}). 
Note that $(-\gamma_1-1/2)^{-1}>(\gamma_2+1)^{-1}$ when $(\gamma_1,\gamma_2)\in \boldsymbol{\Delta}$.
 Therefore in the case of Theorem \ref{Thm:main corner (-1/2,-1)},   the fourth cumulant plays a role in determining the rate of convergence as follows:   if the fourth cumulant converges much slower compared with the third cumulant, that is, if  $L(\gamma_1,\gamma_2)\rightarrow \infty$, then this will slow the rate of convergence in (\ref{eq:d tv goal corner}); if $L(\gamma_1,\gamma_2)$ is asymptotically bounded, then both the third and fourth cumulants behave like $V(\gamma_1,\gamma_2)$.
\end{Rem}

Now we consider the marginal convergence rate in the case of Theorem \ref{Thm:main edge} (see Figure \ref{fig:thm edge}). This theorem involves a non-Gaussian limit. For two random variables $X$ and $Y$ we define the Wasserstein distance between their distributions to be
\[
d_W(X,Y)=\sup_{h\in \mathcal{L}}\{|\E h(X)-\E h(Y)|\},
\]
where $\mathcal{L}$ is the class of $1$-Lipschitz functions ($h\in \mathcal{L}$ if $|h(x)-h(y)|\le |x-y|$). The following result follows from
 \citet{eichelsbacher:thale:2014:malliavin}. 
\begin{Lem}\label{Lem:var gamma}
Let $Y=Z_1Z_2$ where $Z_i$'s are two independent standard normal variables and let $F=I_2(f)$ be an element on the second-order Wiener chaos with $\E F^2=1$.
Then there exists a constant $C>0$ such that
\begin{equation}\label{eq:d_W}
d_W(F,Y)\le C\left(1+\frac{1}{6}\kappa_3(F)^2 -\frac{1}{3}\kappa_4(F)+\frac{1}{120} \kappa_6(F)\right)^{1/2}.
\end{equation}
\end{Lem}
\begin{proof}
By Proposition 1.2(iii) of \citet{gaunt:2014:variance}, the distribution of $Z_1Z_2$ is the symmetric Variance-Gamma $VG(1,0,1,0)$, that is, $VG(2r,0,1/\lambda,0)$ with $r=1/2$ and $\lambda=1$. Inserting these values of $r$ and $\lambda$ in Theorem 5.10(b) of \citet{eichelsbacher:thale:2014:malliavin} gives (\ref{eq:d_W}).
\end{proof}

Using the preceding result, we get the following bound for the convergence rate as $(\gamma_1,\gamma_2)$ approaches the boundary $e_1$.
\begin{Thm}\label{Thm:edge}
Let $Z_{\gamma_1,\gamma_2}=Z_{\gamma_1,\gamma_2}(1)$, and let $Y=Z_1Z_2$ be as in Lemma \ref{Lem:var gamma}.  As
\[(\gamma_1,\gamma_2)\rightarrow (-1/2,\gamma), \quad -1<\gamma<-1/2,
\] we have
\begin{equation}\label{eq:d tv goal boundary}
d_{W}( Z_{\gamma_1,\gamma_2} , Y)= O\left((-\gamma_1-1/2)^{1/2}\right).
\end{equation}
\end{Thm}
\begin{proof}
Following the proof of Theorem \ref{Thm:main edge}, one has by (\ref{eq:cum odd conv}) that as $(\gamma_1,\gamma_2)\rightarrow (-1/2,\gamma)$,  
\begin{equation}\label{eq:kappa 3 e}
\kappa_3(Z_{\gamma_1,\gamma_2})=O\left((-\gamma_1-1/2)^{1/2}\right).
\end{equation}
On the other hand by (\ref{eq:cum even conv}), we have the convergence
$\kappa_m(Z_{\gamma_1,\gamma_2})\rightarrow (m-1)!$ for $m$ even. So
 $\kappa_4(Z_{\gamma_1,\gamma_2})\rightarrow 6$ and $\kappa_6(Z_{\gamma_1,\gamma_2})\rightarrow 120$, and hence 
\[1+\frac{1}{6}\kappa_3(Z_{\gamma_1,\gamma_2})^2 -\frac{1}{3}\kappa_4(Z_{\gamma_1,\gamma_2})+\frac{1}{120} \kappa_6(Z_{\gamma_1,\gamma_2})\rightarrow 1+0-2+1=0.
\]
We thus need  to study  the rate of convergence of the even-order cumulants $\kappa_4$ and $\kappa_6$. It follows from Corollary \ref{Cor:kappa} that
\begin{equation}\label{eq:kappa 4 6 e}
\kappa_4(Z_{\gamma_1,\gamma_2})=6+O\left(-\gamma_1-1/2\right),\quad\kappa_6(Z_{\gamma_1,\gamma_2})=120+O\left(-\gamma_1-1/2\right).
\end{equation}
The proof is concluded by plugging  (\ref{eq:kappa 3 e}) and (\ref{eq:kappa 4 6 e}) in (\ref{eq:d_W}).
\end{proof}

Recently \citet{arras:2016:stein} obtained the rate of convergence  when the limit  is $\sum_{i=1}^q \alpha_i X_i$ where $X_i$'s are  standardized chi-square random variables with one degree of freedom. Appying this result (Theorem 3.1 of \citet{arras:2016:stein}) to the convergence of $(\gamma_1,\gamma_2)\in \boldsymbol{\Delta}$ to the corner $(-1/2,-1/2)$  in the context of Theorem \ref{Thm:main corner (-1/2,-1/2)}, they obtained as $\gamma_1\rightarrow -1/2$ that
\[
d_W(Z_{\gamma_1,\gamma_2}, Y_\rho(1))=O((-\gamma_1-1/2)^{1/2}),
\]
where $Y_\rho(1)$ is as in Theorem \ref{Thm:main corner (-1/2,-1/2)}.
See Example 3.2 of \citet{arras:2016:stein}.
\section{A constructive proof of Theorem \ref{Thm:main edge}}\label{sec:add 3}
The method-of-moments proof of Theorem \ref{Thm:main edge} gives little intuitive insight of the convergence. Motivated by the observation made in Remark \ref{Rem:physical}, we give an alternate proof of Theorem \ref{Thm:main edge}. The proof is based on discretization which removes the singularities at $s=x_1$ and $s=x_2$ of the integrand in (\ref{eq:ori gen Ros proc}), so that one is able to interchange the integration orders between $\int_{\mathbb{R}^2}' \cdot B(dx_1)B(dx_2)$ and $\int_0^t \cdot ds$. Then one uses the triangular approximation described at the end of the proof.

The proof is based on several lemmas.  We use below the notation $(s,x)_N^\gamma$  to denote:
\begin{equation}\label{eq:s,x def}
(s,x)_N^{\gamma}:= \left(\frac{[Ns]-[Nx]+1}{N}\right)^\gamma I\{[Ns]>[Nx]\},\quad \gamma<0.
\end{equation}
Define also
\begin{equation}\label{eq:s,x useful}
[s-x]_N^{\gamma} :=\left(s-x+2/N\right)^{\gamma}I\{s>x+1/N\} \le (s,x)_N^{\gamma} \le  \left(s-x\right)^{\gamma}I\{s>x\}=(s-x)_+^\gamma.
\end{equation}
Let $Z_{\gamma_1,\gamma_2}(t)$ be as in (\ref{eq:ori gen Ros proc}), and let 
\begin{equation}\label{eq:Z_gamma^N}
Z_{\gamma_1,\gamma_2}^N(t)= A_N(\gamma_1,\gamma_2)\int_{\mathbb{R}^2}' \int_0^t (s,x_1)_N^{\gamma_1}(s,x_2)_N^{\gamma_2}ds B(dx_1)B(dx_2),
\end{equation}
where the Brownian measure $B(\cdot)$ is the same as the one defining $Z_{\gamma_1,\gamma_2}(t)$, and where $A_N(\gamma_1,\gamma_2)$ is chosen such that $\E Z_{\gamma_1,\gamma_2}^N(1)^2=1$.  
\begin{Lem}\label{Lem:discrete approx}
For any $t>0$, we have
\begin{equation}\label{eq:limsup triangle}
 \lim_{N\rightarrow\infty} \limsup_{(\gamma_1,\gamma_2)\rightarrow (-1/2,\gamma)} \E \left|Z_{\gamma_1,\gamma_2}(t)-Z_{\gamma_1,\gamma_2}^N(t)\right|^2=0.
\end{equation}
\end{Lem}
\begin{proof}
We take for simplicity that $t=1$, while the other cases can be proved similarly.
Note that
\[
\E \left|Z_{\gamma_1,\gamma_2}(1)-Z_{\gamma_1,\gamma_2}^N(1)\right|^2
= 2-2\E Z_{\gamma_1,\gamma_2}(1)Z_{\gamma_1,\gamma_2}^N(1).
\]
So we need to show that
\begin{equation}\label{eq:dis approx goal}
\lim_{N\rightarrow\infty} \liminf_{(\gamma_1,\gamma_2)\rightarrow (-1/2,\gamma)}\E Z_{\gamma_1,\gamma_2}(1)Z_{\gamma_1,\gamma_2}^N(1)\ge 1.
\end{equation}
Indeed, using the symmetrized kernel in (\ref{eq:gen Rosenblatt proc}), we have
\begin{align}\label{eq:E Z Z^N}
\E Z_{\gamma_1,\gamma_2}(1)Z_{\gamma_1,\gamma_2}^N(1)=&\frac{1}{2}A(\gamma_1,\gamma_2)\frac{1}{2}A_N(\gamma_1,\gamma_2)2!\int_{\mathbb{R}^2}dx_1dx_2  \int_0^1\int_0^1 ds_1ds_2 \notag \\
 & \times[(s_1-x_1)_+^{\gamma_1}(s_1-x_2)_+^{\gamma_2}+(s_1-x_1)_+^{\gamma_2}(s_1-x_2)_+^{\gamma_1}]\notag\\
 &\times[(s_2,x_1)_N^{\gamma_1}(s_2,x_2)_N^{\gamma_2}+(s_2,x_1)_N^{\gamma_2}(s_2,x_2)_N^{\gamma_1}].
\end{align}
By definition,
\begin{align*}
A_N(\gamma_1,\gamma_2)^{-2}=& \frac{1}{2}\int_0^1\int_0^1 ds_1ds_2 \int_{\mathbb{R}^2}dx_1dx_2 
  [(s_1,x_1)_N^{\gamma_1}(s_1,x_2)_N^{\gamma_2}+(s_1,x_1)_N^{\gamma_2}(s_1,x_2)_N^{\gamma_1}]\\
 &\qquad \qquad\qquad\qquad\qquad ~~\times[(s_2,x_1)_N^{\gamma_1}(s_2,x_2)_N^{\gamma_2}+(s_2,x_1)_N^{\gamma_2}(s_2,x_2)_N^{\gamma_1}].
\end{align*}
Applying the second inequality of (\ref{eq:s,x useful}) to (\ref{eq:E Z Z^N}), and using the normalization $A_N(\gamma_1,\gamma_2)$, we have 
\begin{align*}
\E Z_{\gamma_1,\gamma_2}(1)Z_{\gamma_1,\gamma_2}^N(1)\ge  \frac{1}{2}A(\gamma_1,\gamma_2)A_N(\gamma_1,\gamma_2) 2 A_N(\gamma_1,\gamma_2)^{-2}=\frac{A(\gamma_1,\gamma_2)}{A_N(\gamma_1,\gamma_2)}.
\end{align*}
So (\ref{eq:dis approx goal}) follows from the next lemma.
\end{proof}
\begin{Lem}\label{Lem:A ratio}
Let the normalizations $A(\gamma_1,\gamma_2)$ and $A_N(\gamma_1,\gamma_2)$ be as in (\ref{eq:A}) and (\ref{eq:Z_gamma^N}). Then
\begin{equation}\label{eq:A/A_N}
\lim_{N\rightarrow\infty} \lim_{(\gamma_1,\gamma_2)\rightarrow (-1/2,\gamma)} \frac{A(\gamma_1,\gamma_2)}{A_N(\gamma_1,\gamma_2)}=1,
\end{equation}
where $-1<\gamma_1,\gamma_2<-1/2$.
\end{Lem}
\begin{proof}
By the second inequality of (\ref{eq:s,x useful}), we have 
\begin{equation}\label{eq:A_N lemma 1}
A_N(\gamma_1,\gamma_2)^{-2}\le A(\gamma_1,\gamma_2)^{-2}.
\end{equation}
By the first inequality of (\ref{eq:s,x useful}), we have
\begin{align}\label{eq:A_N lemma 2}
A_N(\gamma_1,\gamma_2)^{-2}\ge & \frac{1}{2}\int_0^1\int_0^1 ds_1ds_2 \int_{\mathbb{R}^2}dx_1dx_2 
 \Big([s_1-x_1]_N^{\gamma_1}[s_1-x_2]_N^{\gamma_2}+[s_1-x_1]_N^{\gamma_2}[s_1-x_2]_N^{\gamma_1}\Big)\notag\\
 &\qquad \qquad\qquad\qquad\qquad~~\times\Big([s_2-x_1]_N^{\gamma_1}[s_2-x_2]_N^{\gamma_2}+[s_2-x_1]_N^{\gamma_2}[s_2-x_2]_N^{\gamma_1}\Big)\notag
\\=& P_N(\gamma_1,\gamma_2)+Q_N(\gamma_1,\gamma_2),
\end{align}
where
\begin{align*}
P_N(\gamma_1,\gamma_2)&= 2 \int_{0<s_1<s_2<1} ds_1ds_2   \int_\mathbb{R} [s_1-x_1]_N^{\gamma_1}[s_2-x_1]_N^{\gamma_1} dx_1  \int_{\mathbb{R}} [s_1-x_2]_N^{\gamma_2} [s_2-x_2]_N^{\gamma_2} dx_2,
\end{align*}
and
\begin{align*}
Q_N(\gamma_1,\gamma_2)=2 \int_{0<s_1<s_2<1} ds_1ds_2
\int_\mathbb{R} [s_1-x_1]_N^{\gamma_1}[s_2-x_1]_N^{\gamma_2} dx_1  \int_{\mathbb{R}} [s_1-x_2]_N^{\gamma_2} [s_2-x_2]_N^{\gamma_1} dx_2.
\end{align*}
In the integrals over $\mathbb{R}$, the exponents of $Q_N$ alternate where as those of $P_N$ are the same. 
Note that for $\alpha,\beta \in (-1,-1/2)$ and $0<s_1<s_2<1$, we have
\begin{align}
\int_\mathbb{R} [s_1-x]_N^{\alpha}[s_2-x]_N^{\beta} dx &=\int_{-\infty}^{s_1-1/N} (s_1-x+2/N)^{\alpha}  (s_2-x+2/N)^{\beta}  dx\notag\\
&= \int_0^{\infty} (u+3/N)^{\alpha} (s_2-s_1+u+3/N)^\beta du\label{eq:change of var u}\\
&\le \int_0^\infty u^{\alpha} (u+s_2-s_1)^\beta du = (s_2-s_1)^{\alpha+\beta+1} \B(\alpha+1,-\alpha-\beta-1),\notag
\end{align}
after setting $u=s_1-x-1/N$.
Thus the term $Q_N$ from (\ref{eq:A_N lemma 2}) satisfies
\begin{align}\label{eq:A_N lemma 3}
Q_N(\gamma_1,\gamma_2)\le &2 (2\gamma_1+2\gamma_2+3)^{-1} (2\gamma_1+2\gamma_2+4)^{-1}\notag \\&\times \B(\gamma_1+1,-\gamma_1-\gamma_2-1)\B(\gamma_2+1,-\gamma_1-\gamma_2-1)= O(1).
\end{align}
 as $(\gamma_1,\gamma_2)\rightarrow (-1/2,\gamma)$. The other term $P_N$ in view of (\ref{eq:s,x useful}) and (\ref{eq:change of var u})  becomes
\begin{align*}
P_N(\gamma_1,\gamma_2)=&2 \int_{0<s_1<s_2<1} ds_1ds_2 \int_{0}^{\infty} (u+3/N)^{\gamma_1} (s_2-s_1+u+3/N)^{\gamma_1} du  \notag
\\&\qquad\qquad\qquad\qquad\times \int_{0}^{\infty} (u+3/N)^{\gamma_2} (s_2-s_1+u+3/N)^{\gamma_2} du.
\end{align*}
Now in the second integral, use $(u+3/N)^{\gamma_2}\ge (s_2-s_1+u+3/N)^{\gamma_2}$, and in the third integral, replace $u$ by $u(s_2-s_1)$ and then factor $s_2-s_1$. One gets
\begin{align*}
P_N(\gamma_1,\gamma_2)\ge &2 \int_{0<s_1<s_2<1} ds_1ds_2 \int_0^\infty (s_2-s_1+u+3/N)^{2\gamma_1}du   \notag
\\ &\times(s_2-s_1)^{2\gamma_2+1}\int_0^\infty \left(u+\frac{3}{N(s_2-s_1)}\right)^{\gamma_2}\left(1+u+\frac{3}{N(s_2-s_1)}\right)^{\gamma_2}  du
\end{align*}
Since $\int_0^\infty (s_2-s_1+u+3/N)^{2\gamma_1}du=(-2\gamma_1-1)^{-1}(s_2-s_1+3/N)^{2\gamma_1+1}$, one has
\begin{align}\label{eq:A_N lemma 4}
P_N(\gamma_1,\gamma_2)\ge  & 2 (-2\gamma_1-1)^{-1} \int_{0<s_1<s_2<1}ds_1ds_2 (s_2-s_1+3/N)^{2\gamma_1+1}(s_2-s_1)^{2\gamma_2+1}\notag
 \\& \times
 \int_0^\infty \left(u+\frac{3}{N(s_2-s_1)}\right)^{\gamma_2}\left(u+\frac{3}{N(s_2-s_1)}+1\right)^{\gamma_2}  du=:R_N(\gamma_1,\gamma_2).
\end{align}
As $(\gamma_1,\gamma_2)\rightarrow (-1/2,\gamma)$, we have
\begin{align*}
(-2\gamma_1-1) R_N(\gamma_1,\gamma_2)\rightarrow& 2 \int_{0<s_1<s_2<1}ds_1ds_2 (s_2-s_1)^{2\gamma+1}\\& 
  \times \int_0^\infty \left(u+\frac{3}{N(s_2-s_1)}\right)^{\gamma}\left(u+\frac{3}{N(s_2-s_1)}+1\right)^{\gamma} du.
\end{align*}
As $N\rightarrow\infty$, by the Monotone Convergence Theorem, the right-hand side of the preceding line converges to 
\begin{align*}
&2\int_{0<s_1<s_2<1}ds_1ds_2 (s_2-s_1)^{2\gamma+1}\int_0^\infty u^{\gamma}(u+1)^{\gamma}du= (2\gamma+3)^{-1}(\gamma+1)^{-1}\B(\gamma+1,-2\gamma-1).
\end{align*}
On the other hand, from (\ref{eq:A(gamma)^m}) with $m=2$ we have
\begin{equation}\label{eq:A^2}
A(\gamma_1,\gamma_2)^2\sim (-2\gamma_1-1)(2\gamma+3)(\gamma+1)\B(\gamma+1,-2\gamma-1)^{-1}.
\end{equation}
Hence
\begin{align}\label{eq:A_N lemma 5}
\lim_{N\rightarrow\infty} \lim_{(\gamma_1,\gamma_2)\rightarrow (-1/2,\gamma)}A(\gamma_1,\gamma_2)^2 R_N(\gamma_1,\gamma_2)=1 
\end{align}
Combining  (\ref{eq:A_N lemma 2}), (\ref{eq:A_N lemma 3}), (\ref{eq:A_N lemma 4}) and (\ref{eq:A_N lemma 5}) yields 
\[
 \liminf_{ N\rightarrow\infty} \liminf_{(\gamma_1,\gamma_2)\rightarrow (-1/2,\gamma)} \frac{A(\gamma_1,\gamma_2)^2}{A_N(\gamma_1,\gamma_2)^2}\ge 1,
\]
This with (\ref{eq:A_N lemma 1}) yields (\ref{eq:A/A_N}).
\end{proof}

We will now interchange the integrals $\int_0^t \cdot ds$ and $\int_{\mathbb{R}^2}' \cdot dx_1dx_2$, and write
\begin{align}
Z_{\gamma_1,\gamma_2}^N(t)&=A_N(\gamma_1,\gamma_2) \int_{\mathbb{R}^2}' \left[ \int_0^t (s,x_1)_N^{\gamma_1} (s,x_2)_N^{\gamma_2}  B(dx_1)B(dx_2)ds\right]\notag\\
&=A_N(\gamma_1,\gamma_2) \int_0^t\left[ \int_{\mathbb{R}^2}' (s,x_1)_N^{\gamma_1} (s,x_2)_N^{\gamma_2}  B(dx_1)B(dx_2)\right] ds,\quad \text{a.s.},\label{eq:change int order alt pf}
\end{align}
by the stochastic Fubini theorem (see \citet{pipiras:taqqu:2010:regularization} Theorem 2.1). It applies since
\begin{equation}\label{eq:fubini L2}
\int_0^t \int_{\mathbb{R}^2} \left[(s,x_1)_N^{\gamma_1} (s,x_2)_N^{\gamma_2}\right]^2 dx_1dx_2 ds<\infty. 
\end{equation}
Relation (\ref{eq:fubini L2}) follows from the following lemma.
\begin{Lem}
For any $\gamma\in (-1,-1/2)$, $t>0$ and $N\in \mathbb{Z}_+$, we have 
\[
\sup_{s\in [0,t]} \int_{\mathbb{R}} (s,x)_N^{2\gamma} dx<\infty.
\]
\end{Lem}
\begin{proof}
In view of (\ref{eq:s,x def}), 
\begin{align*}
\int_{\mathbb{R}} (s,x)_N^{2\gamma} dx &=\frac{1}{N}\int_{\mathbb{R}} \left(\frac{[Ns]-[Nx]+1}{N}\right)^{2\gamma} I\{[Ns]>[Nx]\}~ d(Nx)
\\
&=N^{-2\gamma-1} \sum_{-\infty<i<[Ns]}\left([Ns]-i+1\right)^{2\gamma}
= N^{2\gamma-1}\sum_{k=2}^\infty k^{-2\gamma}<\infty
\end{align*}
since $\gamma<-1/2$, where we set $k=[Ns]-i+1$. Since the last expression does not depend on $s$,  the conclusion of the lemma holds.
\end{proof}
By the product formula of Wiener-It\^o integrals (see, e.g., \citet{nourdin:peccati:2012:normal} Theorem 2.7.10),
the process $Z_{\gamma_1,\gamma_2}^N(t)$ in (\ref{eq:change int order alt pf}) can be rewritten as follows:
\begin{align*}
Z_{\gamma_1,\gamma_2}^N(t)=&A_N(\gamma_1,\gamma_2)\\
\times &\int_0^t \left[\int_\mathbb{R} (s,x_1)_N^{\gamma_1} B(dx_1) \int_{\mathbb{R}} (s,x_2)_N^{\gamma_2} B(dx_2) -  \E\int_\mathbb{R} (s,x_1)_N^{\gamma_1} B(dx_1) \int_{\mathbb{R}} (s,x_2)_N^{\gamma_2} B(dx_2)\right]  ds 
\end{align*}
Note that by the scaling property of Brownian motion, for $j=1,2$, 
\begin{align*}
X_{\gamma_j}^N(s)&:=\int_\mathbb{R} (s,x)_N^{\gamma_j} B(dx)=
\int_\mathbb{R} \left(\frac{[Ns]-[Nx]+1}{N}\right)^{\gamma_j}I\{[Ns]>[Nx]\} B(dx)
\\&\overset{f.d.d.}{=} N^{-\gamma_j-1/2}  \sum_{-\infty<i<[Ns]} ([Ns]-i+1)^{\gamma_j}\epsilon_i,
\end{align*}
where $\epsilon_i$'s are i.i.d. standard normal random variables, and $\overset{f.d.d.}{=}$ means equal in finite-dimensional distributions. Hence (recall that the Hurst index $H=\gamma_1+\gamma_2+2$),
\begin{align}\label{eq:Z_gamma rewrite}
Z_{\gamma_1,\gamma_2}^N(t) &\overset{f.d.d.}{=}A_N(\gamma_1,\gamma_2)\int_0^t 、\left[X_{\gamma_1}^N(s)X_{\gamma_2}^N(s) - \E X_{\gamma_1}^N(s)X_{\gamma_2}^N(s)\right] ds\notag\\
&~= A_N(\gamma_1,\gamma_2) N^{-H} \sum_{n=1}^{[Nt]}\left[ Y_{\gamma_1}(n)Y_{\gamma_2}(n)-\E Y_{\gamma_1}(n)Y_{\gamma_2}(n)\right]+R_N(t,\gamma_1,\gamma_2)
\end{align}
where 
\begin{equation}\label{eq:Y stationary}
Y_{\gamma}(n)=  \sum_{-\infty<i<n-1} (n-i)^{\gamma}\epsilon_i= \sum_{i=2}^\infty  i^{\gamma}\epsilon_{n-i}
\end{equation}
is a linear stationary sequence and
\begin{align}\label{eq:R_N remainder}
R_N(t,\gamma_1,\gamma_2)=&A_N(\gamma_1,\gamma_2)N^{-H}\left(Nt-[Nt]\right) \notag\\&
\times \Big(Y_{\gamma_1}([Nt]+1)Y_{\gamma_2}([Nt]+1)-\E Y_{\gamma_1}([Nt]+1)Y_{\gamma_2}([Nt]+1)\Big).
\end{align}
We first show that this preceding remainder term is negligible:
\begin{Lem}\label{Lem:R go to 0}
\begin{equation}\label{eq:R go to 0}
 \lim_{N\rightarrow \infty} \limsup_{(\gamma_1,\gamma_2)\rightarrow (-1/2,\gamma)} \E R_N(t,\gamma_1,\gamma_2)^2=0
\end{equation}
\end{Lem}
\begin{proof}
Since $Nt-[Nt]\le 1$ and $Y_\gamma(n)$ is stationary, we can write
\begin{align*}
\E R_N(t,\gamma_1,\gamma_2)^2 \le N^{-2H} A_N(\gamma_1,\gamma_2)^2 \left[ \E Y_{\gamma_1}(0)^2 Y_{\gamma_2}(0)^2-\left(\E Y_{\gamma_1}(0)Y_{\gamma_2}(0)\right)^2 \right].
\end{align*}
We have
\begin{equation}\label{eq:Y station var cov}
\E Y_{\gamma_1}(0)Y_{\gamma_2}(0)=\sum_{i=2}^\infty i^{\gamma_1+\gamma_2}, \quad \E Y_{\gamma_j}(0)^2= \sum_{i=2}^\infty i^{2\gamma_j},~j=1,2.
\end{equation}
By the diagram formula (see, e.g., \citet{janson:1997:gaussian} Theorem 1.36), we have for jointly centered Gaussian variables $(Y_1,Y_2)$ that $\E Y_1^2 Y_2^2=2\left(\E Y_1 Y_2\right)^2+ \E Y_1^2 \E Y_2^2$.  Expressing this as $\E Y_1^2Y_2^2-\left(\E Y_1 Y_2\right)^2=\left(\E Y_1 Y_2\right)^2+ \E Y_1^2 \E Y_2^2$,
one gets
\begin{align}\label{eq:R_N(t)^2 bound}
\E R_N(t,\gamma_1,\gamma_2)^2 \le N^{-2H}A_N(\gamma_1,\gamma_2)^2 \left[
\left(\sum_{i=2}^\infty i^{\gamma_1+\gamma_2}\right)^2 +\left(\sum_{i=2}^\infty i^{2\gamma_1}\right)\left(\sum_{i=2}^\infty i^{2\gamma_2}\right)\right].
\end{align}
The first and last sums remain bounded as $(\gamma_1,\gamma_2)\rightarrow (-1/2,\gamma)$, but this is not the case for the second sum. Since the function $x^{2\gamma_1}$ is decreasing, we have for any integer $k\ge 0$, 
\begin{align}\label{eq:sum i 2gamma} 
 &(-2\gamma_1-1)^{-1}(k+2)^{2\gamma_1+1}=\int_{2}^\infty (x+k)^{2\gamma_1} dx \le \int_{2}^\infty (x+k)^{\gamma_1} x^{\gamma_1} dx \notag\\
 &\le\sum_{i=2}^\infty (i+k)^{\gamma_1}i^{\gamma_1}\le \sum_{i=2}^\infty i^{2\gamma_1}\le \int_{1}^\infty x^{2\gamma_1} dx =(-2\gamma_1-1)^{-1}.
\end{align}  
In particular,  $\sum_{i=2}^\infty i^{2\gamma_1}$ explodes like $(-2\gamma_1-1)^{-1}$ as $\gamma_1\rightarrow -1/2$. This, however, will be compensated by $A_N(\gamma_1,\gamma_2)^2$, since by   (\ref{eq:A/A_N}) and (\ref{eq:A^2}), we have $A_N(\gamma_1,\gamma_2)\sim A(\gamma_1,\gamma_2)\asymp (-2\gamma_1-1)$ as $(\gamma_1,\gamma_2)\rightarrow (-1/2,\gamma)$. Hence (\ref{eq:R_N(t)^2 bound}) implies 
\[
\limsup_{(\gamma_1,\gamma_2)\rightarrow (-1/2,\gamma)} N^{2H} \E R_N(t,\gamma_1,\gamma_2)^2 <\infty,
\]
which entails (\ref{eq:R go to 0}).
\end{proof}

The following lemma is  key:
\begin{Lem}\label{Lem:asymp indep}
Let $Y_{\gamma}(n)$  be as in (\ref{eq:Y stationary}).
As $(\gamma_1,\gamma_2)\rightarrow (-1/2,\gamma)$, one has the following joint convergence in distribution:
\[
\Big(A(\gamma_1,\gamma_2)Y_{\gamma_1}(n), Y_{\gamma_2}(n) \Big)_{n=1}^{N} \ConvD \Big(\sigma_{\gamma}W, Y_{\gamma}(n) \Big)_{n=1}^N, 
\]
for any $N\in \mathbb{Z}_+$,
where $W$ is a standard normal random variable which is independent of $Y_{\gamma}(n)$, and 
\begin{equation}\label{eq:sigma_gamma}
\sigma_\gamma=(2\gamma+3)^{1/2}(\gamma+1)^{1/2}\B(\gamma+1,-2\gamma-1)^{-1/2}.
\end{equation}
\end{Lem}
\begin{proof}
Since $\Big(A(\gamma_1,\gamma_2)Y_{\gamma_1}(n), Y_{\gamma_2}(n) \Big)_{n=1}^N$ is always a centered and jointly Gaussian vector, we only need to show that its covariance structure converges to that of $\Big(\sigma_{\gamma}W, Y_{\gamma}(n) \Big)_{n=1}^N$. Let us first compute the covariance of $A(\gamma_1,\gamma_2)Y_{\gamma_1}$. By
 (\ref{eq:A^2}) and (\ref{eq:sum i 2gamma}), we have for $m\ge n$ (similarly for $m<n$)
\begin{align*}
\E\left[A(\gamma_1,\gamma_2)Y_{\gamma_1}(n)A(\gamma_1,\gamma_2)Y_{\gamma_1}(m)\right]=&A(\gamma_1,\gamma_2)^2\E\left[Y_{\gamma_1}(n)Y_{\gamma_1}(m)\right]\\
\sim& (2\gamma+3)(\gamma+1) \B(\gamma+1,-2\gamma-1)^{-1} (-2\gamma_1-1) \sum_{i=2}^\infty (i+m-n)^{\gamma_1} i^{\gamma_1}\\\sim & (2\gamma+3)(\gamma+1) \B(\gamma+1,-2\gamma-1)^{-1}=\sigma_\gamma^2. 
\end{align*}
Since the limit is independent of $n$, the limit process is indeed a fixed Gaussian random variable, say $\sigma_\gamma W$.

We now focus on the cross-covariance between  $A(\gamma_1,\gamma_2)Y_{\gamma_1}$ and $Y_{\gamma_2}$.  We have for $m\ge n$ (similarly for $m<n$) that
\begin{align}\label{eq:cov conv zero}
&\E\left[ A(\gamma_1,\gamma_2)Y_{\gamma_1}(n)  Y_{\gamma_2}(m) \right]\notag\\
\sim &[(2\gamma+3)(\gamma+1) \B(\gamma+1,-2\gamma-1)^{-1} (-2\gamma_1-1)]^{1/2} \sum_{i=2}^\infty (i+m-n)^{\gamma_1}i^{\gamma}\rightarrow 0, 
\end{align}
because $\sum_{i=2}^\infty i^{-1/2+\gamma}<\infty$. Thus we have asymptotic independence. 
Finally as $\gamma_2\rightarrow \gamma$,  the covariance structure of  the second term $Y_{\gamma_2}$ converges  to that of $Y_{\gamma}$.
The proof is then complete.
\end{proof}
The following convergence of normalized sum of long-memory linear process to fractional Brownian motion can be found in \citet{giraitis:koul:surgailis:2009:large} Corollary 4.4.1, which was originally due to \citet{davydov:1970:invariance}.
\begin{Lem}\label{Lem:linear to fbm}
Let $Y_\gamma(n)$ be as in (\ref{eq:Y stationary}). Then as $N\rightarrow\infty$
\[
Z_{\gamma}^N(t):=
N^{-\gamma-2/3}\sum_{n=1}^{[Nt]}Y_{\gamma}(n) \ConvFDD \sigma_\gamma^{-1} B_{\gamma+3/2}(t)
\]
where $\sigma_\gamma$ is as in (\ref{eq:sigma_gamma}) and $B_{\gamma+3/2}(t)$ is a standard fractional Brownian motion with Hurst index $\gamma+3/2$.
\end{Lem}

We are now ready to combine the last few lemmas into an alternate proof of Theorem \ref{Thm:main edge}.
\begin{proof}[Proof of Theorem \ref{Thm:main edge}]
Tightness still follows from Lemma \ref{Lem:tight}. To prove the convergence of the finite-dimensional distributions, namely, to prove that 
\[
Z_{\gamma_1,\gamma_2}(t)\ConvFDD W B_{\gamma+3/2}\qquad \text{as $(\gamma_1,\gamma_2)\rightarrow (-1/2,\gamma)$},
\] 
it is sufficient to show that the following triangular approximation relations hold (see, e.g., Lemma 4.2.1 of \citet{giraitis:koul:surgailis:2009:large}):
\begin{equation}\label{eq:physical proof 1}
 \lim_{N\rightarrow\infty} \limsup_{(\gamma_1,\gamma_2)\rightarrow (-1/2,\gamma)} \E \left|Z_{\gamma_1,\gamma_2}(t)- \frac{A(\gamma_1,\gamma_2)}{A_N(\gamma_1,\gamma_2)} [Z_{\gamma_1,\gamma_2}^N(t)-R_N(t,\gamma_1,\gamma_2)]\right|^2=0,
\end{equation}
\begin{equation}\label{eq:physical proof 2}
\frac{A(\gamma_1,\gamma_2)}{A_N(\gamma_1,\gamma_2)}[ Z_{\gamma_1,\gamma_2}^N(t)-R_N(t,\gamma_1,\gamma_2)] \ConvFDD \sigma_\gamma W Z_{\gamma}^N(t)\quad \text{ as $(\gamma_1,\gamma_2)\rightarrow (-1/2,\gamma)$},
\end{equation}
\begin{equation}\label{eq:physical proof 3}
\sigma_\gamma W Z_{\gamma}^N(t) \ConvFDD W B_{\gamma+3/2}(t), \quad \text{as }N\rightarrow\infty.
\end{equation}
The convergence (\ref{eq:physical proof 1}) follows from Lemma \ref{Lem:discrete approx}, Lemma \ref{Lem:A ratio} and Lemma \ref{Lem:R go to 0}. 
For the convergence (\ref{eq:physical proof 2}), we have by  (\ref{eq:Z_gamma rewrite}), Lemma \ref{Lem:asymp indep} and (\ref{eq:cov conv zero}) that
\begin{align*}
\frac{A(\gamma_1,\gamma_2)}{A_N(\gamma_1,\gamma_2)} [Z_{\gamma_1,\gamma_2}^N(t)-R_N(t,\gamma_1,\gamma_2)]&=
 N^{-H} \sum_{n=1}^{[Nt]}\left[ A(\gamma_1,\gamma_2)Y_{\gamma_1}(n)Y_{\gamma_2}(n)-\E A(\gamma_1,\gamma_2) Y_{\gamma_1}(n)Y_{\gamma_2}(n)\right]\\
&\ConvFDD N^{-\gamma-3/2}\sum_{n=1}^{[Nt]}\left[\sigma_\gamma W Y_{\gamma}(n)-0\right] =\sigma_\gamma W Z_{\gamma}^N(t).
\end{align*}
Finally,   (\ref{eq:physical proof 3}) follows from Lemma \ref{Lem:linear to fbm}.
\end{proof}
%\begin{figure}
%\includegraphics[scale=0.7]{}
%\centering
%\caption{$M_3(\gamma)/48$}\label{fig:K3}
%\end{figure}
\noindent\textbf{Acknowledgments.} This work was partially supported by the NSF grants DMS-1007616 and DMS-1309009 at Boston University. We would  like to thank Yin Huang for a suggestion. We also thank   the referees for their helpful comments and suggestions.

\end{document}